\documentclass{amsart}
\usepackage[T1]{fontenc}
\usepackage{amsmath, amssymb, amsthm}

\numberwithin{equation}{section}
\theoremstyle{plain}
\newtheorem{theorem}{Theorem}[section]
\newtheorem{lemma}[theorem]{Lemma}
\newtheorem{proposition}[theorem]{Proposition}
\newtheorem{corollary}[theorem]{Corollary}

\theoremstyle{remark}
\newtheorem{remark}[theorem]{Remark}

\theoremstyle{definition}
\newtheorem{definition}[theorem]{Definition}

\newcommand{\Z}{\mathbb{Z}}
\newcommand{\Q}{\mathbb{Q}}
\newcommand{\C}{\mathbb{C}}
\newcommand{\N}{\mathbb{N}}
\newcommand{\R}{\mathbb{R}}
\newcommand{\U}{\mathbb{U}}
\newcommand{\Na}{\mathrm{N}}

\title[Hal\'asz theorems for Gaussian ideals]{Hal\'asz theorems for Gaussian ideals in sectors and short intervals}
\author{Jan Ku\'s}
\address{Mathematics Institute, University of Warwick, Coventry CV4 7AL, United Kingdom}
\email{jankus493@gmail.com}
\date{\today}

\begin{document}

\begin{abstract}
We prove a quantitative Hal\'asz theorem for multiplicative functions on the nonzero ideals of
$\Z[i]$, with bounds controlled by pretentious distance to the Archimedean characters $\Na^{it}$.
We also prove a sectorial analogue: under angular non-pretentiousness,
the sum of $f$ over ideals lying in a fixed sector is asymptotically given by the expected
proportion of the unrestricted sum.
Finally, under angular non-pretentiousness and a non-degeneracy
condition on conjugate prime pairs, we prove a sectorial short-interval version of the Hal\'asz theorem for
annular sectors whose radial thickness tends to infinity.
The proof of the sectorial short-interval Hal\'asz theorem uses angular Fourier expansion, norm-compression to multiplicative functions on $\N$,
and a theorem of Mangerel.
\end{abstract}
\maketitle

\tableofcontents
\bigskip

\section{Introduction}

Hal\'asz's theorem bounds the mean value of a multiplicative function in terms of how closely it
pretends to be one of the Archimedean characters $n^{it}$. In the pretentious framework of Granville--Soundararajan, this
closeness is measured by a prime-weighted distance; see \cite{GS03,GHS19,GS} for background and
refinements in the classical setting.

In this paper we prove corresponding results over the Gaussian integers.
We work with multiplicative functions on the nonzero ideals of $\Z[i]$, ordered by norm.
Using ideals rather than elements removes the ambiguity coming from the unit group.
At the same time, by choosing for each nonzero ideal its unique canonical generator in the first
quadrant, one can still recover angular information.

We prove three results. The first is a quantitative Hal\'asz theorem for multiplicative functions
on Gaussian ideals. The second is a sectorial version: under angular non-pretentiousness,
the sum over a fixed sector contributes the expected proportion of the unrestricted sum.
The third is a sectorial short-interval Hal\'asz theorem under angular non-pretentiousness and a
non-degeneracy condition on conjugate prime pairs.

Throughout, let $\mathcal I$ denote the set of nonzero ideals of $\Z[i]$. 
For $\mathfrak a\in\mathcal I$ write $\Na\mathfrak a$ for its norm. Write  $\Re z$ for the real part of $z$, $\overline z$ for complex conjugation, and
$
\overline{\mathfrak a}:=\{\overline z:z\in\mathfrak a\}
$
for the conjugate ideal.
A function $f:\mathcal I\to\C$ is called multiplicative if $f(\Z[i])=1$ and
$f(\mathfrak a\mathfrak b)=f(\mathfrak a)f(\mathfrak b)$ whenever $(\mathfrak a,\mathfrak b)=\Z[i]$.
It is completely multiplicative if
$f(\mathfrak a\mathfrak b)=f(\mathfrak a)f(\mathfrak b)$ for all
$\mathfrak a,\mathfrak b\in\mathcal I$.

The proof of the quantitative ideal Hal\'asz theorem follows the
Perron--triple-convolution argument of Granville--Harper--Soundararajan, in the
pretentious framework of Granville--Soundararajan; the main adaptation is the
corresponding short-interval estimate for the ideal von Mangoldt function.
The sectorial Hal\'asz theorem is proved by expanding the sector
indicator into angular Fourier modes and applying Theorem~\ref{thm:halasz_1bounded_mult} to the functions $f\lambda_m$.
In the proof of the sectorial short-interval Hal\'asz theorem, we combine the same Fourier expansion with norm-compression to
multiplicative functions on $\N$ and then apply a theorem of Mangerel.

The closest precedent for Theorem~\ref{thm:halasz_1bounded_mult} is the quantitative
Hal\'asz theorem of Granville, Harper and Soundararajan over $\N$, whose proof also
underlies our argument; their function-field work gives the analogous result over
$\mathbb F_q[x]$ \cite{GHS19,GHSFF}. Earlier qualitative analogues of Hal\'asz's theorem in
much more general settings were obtained by Lucht and Reifenrath \cite{LR01}, and by
Debruyne, Maes and Vindas \cite{DMV20}; these results cover abstract semigroup and Beurling
settings, but do not provide a quantitative bound of the type proved in
Theorem~\ref{thm:halasz_1bounded_mult}. In the Gaussian setting, Donoso, Le, Moreira and
Sun \cite{DLMS24} proved a Gaussian-integer analogue of Wirsing's theorem for bounded
completely multiplicative functions.  More recently, Donoso, Ferr\'e Moragues, Koutsogiannis and
Sun \cite{DFKS26} used the qualitative Hal\'asz theorem of Lucht and Reifenrath as an input
in their study of partition regularity over imaginary quadratic rings.

\subsection{Hal\'asz theorems over $\Z[i]$}

We begin with the unrestricted mean-value problem for multiplicative functions on Gaussian ideals.
For the first theorem only the $1$-bounded pretentious distance is needed.
For $x\ge 3$ and $t\in\R$, define
\[
\mathbb D(f,\Na^{it};x)^2
:=\sum_{\Na\mathfrak p\le x}
\frac{1-\Re\!\bigl(f(\mathfrak p)(\Na\mathfrak p)^{-it}\bigr)}{\Na\mathfrak p},
\]
where the sum runs over prime ideals $\mathfrak p\subset\Z[i]$, and set
\[
M_{\mathrm{pret}}(x):=\min_{|t|\le \log x}\mathbb D(f,\Na^{it};x)^2.
\]
Here $\Na^{it}$ denotes the completely multiplicative function on ideals given on prime ideals by
\[
(\Na^{it})(\mathfrak p)=(\Na\mathfrak p)^{it}.
\]

\begin{theorem}[Hal\'asz pretentious form for $1$-bounded multiplicative functions]
\label{thm:halasz_1bounded_mult}
Let $f:\mathcal I\to\C$ be multiplicative on nonzero ideals of $\Z[i]$ and assume that
$|f(\mathfrak a)|\le 1$ for all $\mathfrak a\in\mathcal I$.
Then
\[
\sum_{\Na\mathfrak a\le x} f(\mathfrak a)
\ \ll\
(1+M_{\mathrm{pret}}(x))e^{-M_{\mathrm{pret}}(x)}\,x
\ +\ \frac{x}{\log x}\,\log\log x,
\]
for all $x\ge 3$, with an absolute implied constant.
\end{theorem}

\begin{remark}
Under the hypotheses of Theorem~\ref{thm:halasz_1bounded_mult}, let $0<A\le 1$ and assume that
\[
M_{\mathrm{pret}}(x)\ge 2A\log\log x
\]
for all sufficiently large $x$. Then
\[
\sum_{\Na\mathfrak a\le x} f(\mathfrak a)
\ \ll_A\
\frac{x\log\log x}{(\log x)^A}.
\]
In particular, if $\lim_{x\to\infty}M_{\mathrm{pret}}(x)=+\infty$, then
\[
\sum_{\Na\mathfrak a\le x} f(\mathfrak a)=o(x).
\]
\end{remark}

We next state the more flexible quantitative ideal Hal\'asz theorem from which
Theorem~\ref{thm:halasz_1bounded_mult} is deduced. This is the Gaussian-ideal analogue of the
quantitative Hal\'asz theorem of Granville, Harper and Soundararajan
\cite[Theorem~1.1 and Corollary~1.2]{GHS19}.
For this it is convenient to allow a general
parameter $\kappa\ge 1$.
The following analytic form is the version proved directly; it is later converted to the
pretentious form by an Euler product estimate.
Write \[
\U:=\{z\in\C: |z|\le 1\}.
\]
\begin{definition}[Pretentious distance]\label{def:pret_distance}
Let $\kappa\ge 1$, let $f:\mathcal I\to\C$ and $g:\mathcal I\to\U$ be multiplicative, and assume
that
\[
|f(\mathfrak p)|\le \kappa
\]
for every prime ideal $\mathfrak p\subset\Z[i]$. For $x\ge 2$, define
\[
\mathbb D_\kappa(f,g;x)^2
:= \sum_{\Na\mathfrak p\le x}
\frac{\kappa-\Re\!\left(f(\mathfrak p)\overline{g(\mathfrak p)}\right)}{\Na\mathfrak p},
\]
where the sum is over prime ideals $\mathfrak p\subset\Z[i]$.
For $\kappa=1$ we abbreviate
\[
\mathbb D(f,g;x):=\mathbb D_1(f,g;x).
\]
\end{definition}
\begin{theorem}[Quantitative ideal Hal\'asz theorem]\label{thm:ideal-halasz-M}
Let $\kappa\ge 1$ and let $f:\mathcal I\to\C$ be multiplicative. Define
\[
F(s):=\sum_{\mathfrak a\in\mathcal I}\frac{f(\mathfrak a)}{(\Na\mathfrak a)^s}\qquad(\Re(s)>1).
\]
Assume that the associated von Mangoldt coefficients $\Lambda_f$ of $f$
(as in Definition~\ref{def:Lambda_f_coeff}) satisfy
\[
|\Lambda_f(\mathfrak a)|\le \kappa\,\Lambda(\mathfrak a)
\]
for all $\mathfrak a\in\mathcal I$, where $\Lambda$ is the ideal von Mangoldt function
from Definition~\ref{def:ideal-vonmangoldt}. For $x\ge 3$, define $M=M(x)$ by
\[
\max_{|t|\le (\log x)^{\kappa}}
\left|\frac{F\!\left(1+\frac1{\log x}+it\right)}{1+\frac1{\log x}+it}\right|
\ :=\ e^{-M}(\log x)^{\kappa},
\]
and put
\[
M_+(x):=\max(M(x),0).
\]
Then for all $x\ge 3$,
\[
\sum_{\Na\mathfrak a\le x} f(\mathfrak a)
\ \ll_{\kappa}\
(1+M_+(x))e^{-M_+(x)}\,x(\log x)^{\kappa-1}
\ +\
\frac{x}{\log x}(\log\log x)^{\kappa}.
\]
\end{theorem}

\subsection{Sectorial results}

We now turn to angular restrictions. For an interval $J\subset[0,\pi/2)$, we compare the
unrestricted sum with the sum over ideals whose canonical generators have argument in $J$.
The first theorem in this direction is a sectorial Hal\'asz theorem, and the second is its
short-interval analogue.

\begin{definition}[Angular notation]\label{def:angular_notation}
For each nonzero ideal $\mathfrak a\subset\Z[i]$, let $z_{\mathfrak a}\in\Z[i]$ be its unique
generator with $0\le \arg z_{\mathfrak a}<\pi/2$. Set
\[
\arg(\mathfrak a):=\arg z_{\mathfrak a},
\qquad
\lambda_m(\mathfrak a):=e^{4im\arg(\mathfrak a)}\qquad(m\in\Z).
\]
\end{definition}
Since $\arg(\mathfrak a\mathfrak b)\equiv \arg(\mathfrak a)+\arg(\mathfrak b)\pmod{\pi/2}$, each
$\lambda_m$ is completely multiplicative.
For any arithmetic function $F$ on $\mathcal I$, for $x\ge 1$ and $h\ge 1$, define
\[
S_F(x):=\sum_{\Na\mathfrak a\le x}F(\mathfrak a),
\qquad
S_F(x;h):=\sum_{x<\Na\mathfrak a\le x+h}F(\mathfrak a).
\]
If $J=[\theta_1,\theta_2)\subset[0,\pi/2)$ is an interval, set
\[
S_{F,J}(x):=\sum_{\substack{\Na\mathfrak a\le x\\ \arg(\mathfrak a)\in J}}F(\mathfrak a),
\qquad
S_{F,J}(x;h):=\sum_{\substack{x<\Na\mathfrak a\le x+h\\ \arg(\mathfrak a)\in J}}F(\mathfrak a),
\qquad
\delta_J:=\frac{|J|}{\pi/2}.
\]

\begin{definition}[Pretentious parameters for angular characters]\label{def:M_m_1bdd_fixed}
Assume $f:\mathcal I\to\C$ is multiplicative with $|f(\mathfrak a)|\le 1$.
For $m\in\Z$, $x\ge 3$, and $t\in\R$ define
\[
\mathbb D_m(f;t;x)^2
:=\sum_{\Na\mathfrak p\le x}\frac{1-\Re\!\Bigl(f(\mathfrak p)\lambda_m(\mathfrak p)(\Na\mathfrak p)^{-it}\Bigr)}{\Na\mathfrak p},
\]
where the sum is over prime ideals $\mathfrak p\subset\Z[i]$.
Equivalently,
\[
\mathbb D_m(f;t;x)^2=\mathbb D(f\lambda_m,\Na^{it};x)^2.
\]
Set
\[
M_m(x):=\min_{|t|\le \log x}\mathbb D_m(f;t;x)^2,
\qquad
\widetilde M_m(x):=\min_{|t|\le 2x}\mathbb D_m(f;t;x)^2.
\]
Since $\lambda_0\equiv 1$, we have
\[
M_0(x)=M_{\mathrm{pret}}(x).
\]
The quantities $M_m(x)$ measure how closely $f\lambda_m$ can pretend to be $\Na^{it}$,
equivalently how closely $f$ can pretend to be $\lambda_{-m}(\cdot)\,\Na(\cdot)^{it}$.
The quantity $\widetilde M_m$ is used only in the short-interval Hal\'asz theorem.
\end{definition}

\begin{theorem}[Sectorial Hal\'asz theorem under angular non-pretentiousness]\label{thm:sectorial_cancellation_1bdd_fixed}
Let $f:\mathcal I\to\C$ be multiplicative with $|f(\mathfrak a)|\le 1$ for all ideals $\mathfrak a$.
Let $J=[\theta_1,\theta_2)\subset [0,\pi/2)$ and let $\delta_J=|J|/(\pi/2)$.
Assume that
\[
    \lim_{x\to\infty} M_m(x)=+\infty 
\quad\text{for every integer }m\ne 0.
\]
Then
\[
S_{f,J}(x)=\delta_J S_f(x)+o(x).
\]
If in addition $\lim_{x\to\infty} M_0(x)=+\infty$, then Theorem~\ref{thm:halasz_1bounded_mult} gives
$S_f(x)=o(x)$, and hence $S_{f,J}(x)=o(x)$.
\end{theorem}

\begin{remark}
We can deduce a quantitative version of this theorem as well. By Theorem~\ref{thm:sectorial_halasz_1bdd_fixed}, if there exist $0<A\le 1$ and an integer-valued function
$T=T(x)\ge 2$ with $\lim_{x\to\infty}T(x)=+\infty$ such that
\[
\min_{1\le |m|\le T(x)}M_m(x)\ge 2A\log\log x
\]
for all sufficiently large $x$, then
\[
\bigl|S_{f,J}(x)-\delta_J S_f(x)\bigr|
\ll_A
\frac{x\log(T(x)+1)\log\log x}{(\log x)^{A}}
+\frac{x\log(T(x)+1)}{T(x)}
+x^{1/2}.
\]
The function $T(x)$ can be freely chosen, subject to this hypothesis. Taking $T(x)=\lfloor(\log x)^{A}\rfloor$ gives
\[
\bigl|S_{f,J}(x)-\delta_J S_f(x)\bigr|\ll_A\frac{x(\log\log x)^2}{(\log x)^{A}},
\]
so it suffices that the hypothesis holds for $1\le|m|\le(\log x)^{A}$.
\end{remark}

\begin{remark}
Since $z^4$ is invariant under multiplication by units, we have
$\lambda_m(\mathfrak a)=(z/|z|)^{4m}$ for every generator $z$ of $\mathfrak a$. The functions
$\lambda_m$ are Hecke characters of $\Q(i)$, and Hecke used them to prove that Gaussian
primes are equidistributed in sectors \cite{Hecke20}. Theorem~\ref{thm:sectorial_cancellation_1bdd_fixed}
can be viewed as a weighted version of the equidistribution of $\arg(\mathfrak a)$, in which
the hypothesis asks that $f$ does not pretend to be $\lambda_{-m}\Na^{it}$ for any $m\ne0$.
\end{remark}

The second sectorial result is a short-interval version of this theorem.
Its proof combines angular Fourier decomposition with norm-compression to multiplicative
functions on $\N$, and then applies a theorem of Mangerel.

\begin{theorem}[Sectorial short-interval Hal\'asz theorem]\label{thm:sectorial_MR}
Let $f:\mathcal I\to\C$ be multiplicative with $|f(\mathfrak a)|\le 1$.
Let $h=h(X)$ satisfy $h=o(X)$ and $h/\sqrt X\to\infty$ as $X\to\infty$.
Let $J=J(X)\subset[0,\pi/2)$ be any interval.

Assume that, for every integer $m\ne 0$, the following two hypotheses hold.

\textup{(H1)} There exist $A_m>0$ and $C_m\ge0$ such that, for all $2\le z\le w$,
\[
\sum_{\substack{z<\Na\mathfrak p\le w\\ \mathfrak p\ne \overline{\mathfrak p}}}
\frac{\left|f(\mathfrak p)\lambda_m(\mathfrak p)+f(\overline{\mathfrak p})\lambda_m(\overline{\mathfrak p})\right|}{\Na\mathfrak p}
\ \ge\
A_m\sum_{\substack{z<\Na\mathfrak p\le w\\ \mathfrak p\ne \overline{\mathfrak p}}}\frac{1}{\Na\mathfrak p}
\ -\ \frac{C_m}{\log z},
\]
where the sums are over prime ideals $\mathfrak p\subset \Z[i]$ distinct from their conjugates.
The constants $A_m$ and $C_m$ may depend on $f$ and $m$.

\textup{(H2)} One has
\[
    \lim_{X\to\infty} \widetilde M_m(X)=+\infty .
\]

Then
\[
\frac{2}{X}\int_{X/2}^{X}
\left|\frac{S_{f,J}(x;h)-\delta_J S_f(x;h)}{h}\right|^2\,dx
\ =\ o(1),
\]
as $X\to\infty$.
\end{theorem}

\begin{remark}
The sum $S_{f,J}(x;h)$ runs over ideals whose canonical generators lie in the annular sector
\[
\{re^{i\theta}: \theta\in J,\ \sqrt{x}<r\le \sqrt{x+h}\}.
\]
Its radial thickness is
\[
\sqrt{x+h}-\sqrt{x}=\frac{h}{\sqrt{x+h}+\sqrt{x}}\asymp \frac{h}{\sqrt{x}}.
\]
Thus the hypothesis $h/\sqrt X\to\infty$ means that this thickness tends to infinity,
with no lower rate assumed.
\end{remark}

\begin{remark}
Hypothesis \textup{(H1)} is a non-degeneracy condition on the paired values above conjugate pairs
of prime ideals. It prevents excessive cancellation within each pair and is the input used to
verify the prime lower-bound hypothesis in \cite[Theorem~1.7]{Mangerel} after norm-compression.
\end{remark}

\begin{remark}
The conclusion of Theorem~\ref{thm:sectorial_MR} implies that for all $x\in[X/2,X]$ outside a
set of Lebesgue measure $o(X)$,
\[
\frac{S_{f,J}(x;h)}{h}
=
\delta_J\,\frac{S_f(x;h)}{h}
\ +\ o(1).
\]
\end{remark}

\begin{remark}\label{rem:sectorial_MR_unrestricted}
If, in addition, the analogues of \textup{(H1)} and \textup{(H2)} hold for $m=0$, that is,
if \textup{(H1)} holds with $m=0$ and
\[
    \lim_{X\to\infty}\widetilde M_0(X)=+\infty,
\]
then
\[
\frac{2}{X}\int_{X/2}^{X}
\left|\frac{S_f(x;h)}{h}\right|^2\,dx
\ =\ o(1),
\]
and hence, for all $x\in[X/2,X]$ outside a set of Lebesgue measure $o(X)$,
\[
\frac{S_f(x;h)}{h}=o(1).
\]
\end{remark}

\begin{remark}
The proof of Theorem~\ref{thm:sectorial_MR} also gives a quantitative bound. By
Proposition~\ref{prop:sector_decomp_window}, for every integer $T\ge2$ the square root of
the left-hand side in Theorem~\ref{thm:sectorial_MR} is
\[
\ll\log(T+1)\max_{1\le|m|\le T}E_m(X)+\frac{\log(T+1)}{T}+\frac{\sqrt X}{h},
\]
with $E_m(X)$ as in Proposition~\ref{prop:fixed_flambda_short_interval}. Thus a
quantitative form of Theorem~\ref{thm:sectorial_MR} reduces to quantitative bounds for the
fixed-mode errors $E_m(X)$, uniformly for $1\le|m|\le T$. We do not pursue this here.
\end{remark}

\subsection{Organization of the paper}
Section~\ref{sec:prelim} collects the ideal-theoretic and analytic preliminaries, including
norm-compression, Perron's formula for ideals, and the prime-ideal estimates used later.
Section~\ref{sec:ideal-halasz-proof} proves the quantitative ideal Hal\'asz theorem.
Section~\ref{sec:applications} develops the pretentious reformulations and deduces the $1$-bounded
theorem. Section~\ref{sec:sectorial_halasz} proves a quantitative sectorial Hal\'asz theorem and
deduces the qualitative corollary stated above. Section~\ref{sec:sectorial_MR} proves the
sectorial short-interval Hal\'asz theorem. Appendix~\ref{app:aux} collects auxiliary prime-ideal
and Euler-product estimates used in Section~\ref{sec:ideal-halasz-proof}.

\section{Preliminaries}\label{sec:prelim}

In this section we collect the ideal-theoretic notation and analytic inputs used throughout.
We write
\[
\zeta_{\Q(i)}(s):=\sum_{\mathfrak a\in\mathcal I}\frac{1}{(\Na\mathfrak a)^s}
=\prod_{\mathfrak p}\left(1-\frac{1}{(\Na\mathfrak p)^s}\right)^{-1}
\qquad(\Re(s)>1).
\]

\subsection{Prime ideals and Dirichlet convolution on $\Z[i]$}

We recall the factorization of rational primes in $\Z[i]$:
\begin{enumerate}
\item $2$ is ramified:
\[
(2)=\mathfrak p_2^2
\qquad\text{with}\qquad
\Na\mathfrak p_2=2.
\]
\item If $p$ is a positive rational prime with $p\equiv 1\pmod 4$, then $p$ splits:
\[
(p)=\mathfrak p\,\overline{\mathfrak p},
\qquad
\mathfrak p\neq \overline{\mathfrak p},
\qquad
\Na\mathfrak p=\Na\overline{\mathfrak p}=p.
\]
\item If $p$ is a positive rational prime with $p\equiv 3\pmod 4$, then $p$ is inert:
\[
(p)\ \text{is prime},\qquad \Na((p))=p^2.
\]
\end{enumerate}
Thus the prime ideals are as follows: the ramified prime above $2$, of norm $2$; two prime
ideals of norm $p$ for each positive rational prime $p\equiv 1\pmod 4$; and one prime ideal of norm $p^2$
for each positive rational prime $p\equiv 3\pmod 4$.

\begin{definition}[Dirichlet convolution on ideals]\label{def:ideal-convolution}
If $f,g:\mathcal I\to\C$ are arithmetic functions on nonzero ideals, define
\[
(f*g)(\mathfrak a):=\sum_{\mathfrak d\mid \mathfrak a} f(\mathfrak d)\,g(\mathfrak a/\mathfrak d).
\]
\end{definition}

\subsection{Prime-ideal estimates}

Because
\[
\zeta_{\Q(i)}(s)=\zeta(s)L(s,\chi_{-4}),
\]
estimates for prime ideals in $\Z[i]$ reduce to the classical prime number theorem and Mertens-type
estimates for primes in the residue classes $1$ and $3 \pmod 4$. We record the two consequences
needed later; see, for instance, \cite[Ch.~5]{IK}.

\begin{definition}[Ideal von Mangoldt function]\label{def:ideal-vonmangoldt}
Define
\[
\Lambda(\mathfrak a)=
\begin{cases}
\log \Na\mathfrak p & \text{if }\mathfrak a=\mathfrak p^m \text{ for some prime ideal }\mathfrak p,\\
0 & \text{otherwise.}
\end{cases}
\]
Also write
\[
\psi_{\Q(i)}(x):=\sum_{\Na\mathfrak a\le x}\Lambda(\mathfrak a).
\]
\end{definition}

\begin{lemma}[prime ideal theorem bound]\label{lem:PIT}
    For all $x\ge 2$
 \[\psi_{\Q(i)}(x)\ll x.\]
\end{lemma}
\begin{proof}
Using the factorization of rational primes in \(\mathbb Z[i]\), we have
\[
\psi_{\mathbb Q(i)}(x)
\ll
1
+2\sum_{\substack{p^k\le x\\ p\equiv 1\pmod 4}}\log p
+2\sum_{\substack{p^{2k}\le x\\ p\equiv 3\pmod 4}}\log p
+\log x .
\]
The first sum is \(O(\psi(x))\), and the second is \(O(\psi(\sqrt x))\), where
\(\psi\) is the classical Chebyshev function. Since \(\psi(u)\ll u\), it follows that
\[
\psi_{\mathbb Q(i)}(x)\ll x+\sqrt x+\log x\ll x .
\]
\end{proof}
\begin{lemma}\label{lem:Mertens}
For $x\ge 3$,
\[
\sum_{\Na\mathfrak p\le x}\frac{1}{\Na\mathfrak p}=\log\log x+O(1).
\]
Consequently, for $3\le y\le x$,
\[
\sum_{y<\Na\mathfrak p\le x}\frac{1}{\Na\mathfrak p}=\log\log x-\log\log y+O(1),
\]
and
\[
\prod_{\Na\mathfrak p\le x}\left(1-\frac{1}{\Na\mathfrak p}\right)^{-1}\ll \log x.
\]
In particular, for any $\kappa\ge 1$ and $y\ge 3$,
\[
\prod_{\Na\mathfrak p\le y}\left(1-\frac{1}{\Na\mathfrak p}\right)^{-\kappa}\ll_{\kappa} (\log y)^{\kappa}.
\]
\end{lemma}
\begin{proof}
By the splitting law,
\[
\sum_{\mathrm N\mathfrak p\le x}\frac1{\mathrm N\mathfrak p}
=
2\sum_{\substack{p\le x\\ p\equiv1\pmod4}}\frac1p
+
\sum_{\substack{p^2\le x\\ p\equiv3\pmod4}}\frac1{p^2}
+O(1).
\]
The second sum is \(O(1)\), while Mertens' theorem in arithmetic progressions gives
\[
\sum_{\substack{p\le x\\ p\equiv1\pmod4}}\frac1p
=
\frac12\log\log x+O(1).
\]
This proves the first assertion. The estimate for \(y<N\mathfrak p\le x\) follows by
subtraction, and the Euler product bound follows by taking logarithms and using
\(\log(1-u)^{-1}=u+O(u^2)\).
\end{proof}
\subsection{A Brun--Titchmarsh bound modulo $4$}

\begin{lemma}\label{lem:BT_mod4}
Let $U\ge 3$ and $2\le H\le U$.
Then
\[
\pi(U+H;4,1)-\pi(U;4,1)\ll \frac{H}{\log(2H)}.
\]
Consequently,
\[
\sum_{\substack{U<p\le U+H\\ p\equiv 1\ (\mathrm{mod}\ 4)}}\log p
\ll \frac{H\log U}{\log(2H)}.
\]
\end{lemma}

\begin{proof}
By the Brun--Titchmarsh inequality of Montgomery and Vaughan \cite{MontgomeryVaughan1973},
\[
\pi(U+H;4,1)-\pi(U;4,1)\le\frac{2H}{\varphi(4)\log(H/4)}=\frac{H}{\log(H/4)}
\qquad(H>4).
\]
For $H\ge16$ we have $\log(H/4)\ge\frac12\log H\ge\frac14\log(2H)$, which gives the first
bound. For $2\le H<16$ the left-hand side is at most $H+1\ll H/\log(2H)$. The second bound
follows since $\log p\le\log(2U)\ll\log U$ for every prime $p\in(U,U+H]$.
\end{proof}

\subsection{Norm-compression to arithmetic functions}

\begin{definition}[Norm-compression]\label{def:norm-compression}
For any function $h:\mathcal I\to\C$, define $h^*:\N\to\C$ by
\[
h^*(n):=\sum_{\Na\mathfrak a=n} h(\mathfrak a)
\qquad (n\ge 1).
\]
\end{definition}

\begin{lemma}\label{lem:norm-compression}
Let $h:\mathcal I\to\C$. Then:
\begin{enumerate}
\item[(i)] for every $x\ge 1$,
\[
\sum_{\Na\mathfrak a\le x} h(\mathfrak a)=\sum_{n\le x} h^*(n);
\]
\item[(ii)] if 
\[
\sum_{\mathfrak a\in\mathcal I}\frac{h(\mathfrak a)}{(\Na\mathfrak a)^s}
\]
converges absolutely, then
\[
\sum_{\mathfrak a\in\mathcal I}\frac{h(\mathfrak a)}{(\Na\mathfrak a)^s}
=\sum_{n\ge 1}\frac{h^*(n)}{n^s};
\]
\item[(iii)] if $h$ is multiplicative on ideals, then $h^*$ is multiplicative on $\N$.
\end{enumerate}
\end{lemma}

\begin{proof}
The first two identities follow by grouping terms according to the integer norm $n=\Na\mathfrak a$.
If $(m,n)=1$, then every ideal of norm $mn$ factors uniquely as $\mathfrak a\mathfrak b$ with
$\Na\mathfrak a=m$ and $\Na\mathfrak b=n$. Hence, if $h$ is multiplicative,
\[
h^*(mn)=\sum_{\Na\mathfrak c=mn} h(\mathfrak c)
=\sum_{\substack{\Na\mathfrak a=m\\ \Na\mathfrak b=n}} h(\mathfrak a\mathfrak b)
=\sum_{\substack{\Na\mathfrak a=m\\ \Na\mathfrak b=n}} h(\mathfrak a)h(\mathfrak b)
=h^*(m)h^*(n).
\]
\end{proof}

\subsection{Perron's formula}

We use Perron's formula for Dirichlet series whose coefficients are indexed by ideals and
ordered by norm. This is obtained by grouping the coefficients according to their integer norm,
as in Lemma~\ref{lem:norm-compression}, and then applying the usual Perron formula over
\(\N\). We record both the infinite and truncated forms: the former is used for exact integral
identities, while the latter is used for quantitative truncations.

\begin{theorem}\label{thm:perron}
Let \(A(s)=\sum_{\mathfrak a} c(\mathfrak a)(\Na\mathfrak a)^{-s}\) be absolutely convergent
for \(\Re s>\sigma_0\). Then, for \(x\ge 1\) not equal to the norm of an ideal and any
\(\sigma>\sigma_0\),
\[
\sum_{\Na\mathfrak a\le x} c(\mathfrak a)
=
\frac{1}{2\pi i}\int_{\sigma-i\infty}^{\sigma+i\infty} A(s)\frac{x^s}{s}\,ds.
\]
\end{theorem}

\begin{proof}
For \(n\ge 1\), set
\[
c^*(n):=\sum_{\Na\mathfrak a=n}c(\mathfrak a).
\]
By Lemma~\ref{lem:norm-compression},
\[
A(s)=\sum_{n\ge 1}\frac{c^*(n)}{n^s}
\]
in the half-plane of absolute convergence. The result is the usual Perron formula applied
to the coefficients \(c^*(n)\), rewritten in terms of ideals.
\end{proof}

\begin{lemma}\label{lem:perron-quant}
Let \(A(s)=\sum_{\mathfrak a} c(\mathfrak a)(\Na\mathfrak a)^{-s}\), and suppose that
\[
\sum_{\mathfrak a}|c(\mathfrak a)|(\Na\mathfrak a)^{-\sigma}<\infty
\]
for some \(\sigma>0\). Let \(x\ge 2\), \(T\ge 1\). Then
\[
\begin{aligned}
\sum_{\Na\mathfrak a\le x}c(\mathfrak a)
&=
\frac{1}{2\pi i}\int_{\sigma-iT}^{\sigma+iT}A(s)\frac{x^s}{s}\,ds \\
&\quad+
O\!\left(
\sum_{\mathfrak a}|c(\mathfrak a)|
\left(\frac{x}{\Na\mathfrak a}\right)^\sigma
\min\!\left\{1,\frac{1}{T|\log(x/\Na\mathfrak a)|}\right\}
\right).
\end{aligned}
\]
\end{lemma}

\begin{remark}\label{rem:perron-boundary}
In Lemma~\ref{lem:perron-quant}, the factor
\[
 \min\left\{1,\frac{1}{T|\log(x/\Na\mathfrak a)|}\right\}
\]
is interpreted as \(1\) when \(x=\Na\mathfrak a\). With this convention the
stated truncated formula is valid for all \(x\geq 2\). Equivalently, in applications where
one wants to avoid boundary values, if \(x\) is the norm of an ideal one may apply the
lemma with \(x+\tfrac12\) in place of \(x\). Since ideal norms are integers, this does
not change the sum \(\sum_{\Na\mathfrak a\leq x}c(\mathfrak a)\), and the resulting
change in the error terms is absorbed in the estimates used below.
\end{remark}

\begin{proof}
Let
\[
c^*(n):=\sum_{\Na\mathfrak a=n}c(\mathfrak a).
\]
Apply the classical truncated Perron formula to the Dirichlet series
\[
A(s)=\sum_{n\ge1}\frac{c^*(n)}{n^s}.
\]
See \cite[\S2.4.3, (2.4.5)]{GS}. Rewriting the result in terms of ideals gives the
stated formula.
\end{proof}

\section{Proof of the ideal Hal\'asz theorem}\label{sec:ideal-halasz-proof}

In this section we prove Theorem~\ref{thm:ideal-halasz-M}. The proof follows the
Perron--triple-convolution argument of Granville--Harper--Soundararajan, in the
form presented in~\cite[\S\S2.4--2.5]{GS}. Since the arguments of~\cite{GS,GHS19}
are written over $\N$, and the passage to ideals affects several of the steps, we give
the proofs in full. The auxiliary prime-ideal sums and Euler-product bounds are
collected in Appendix~\ref{app:aux}.

We first define the associated von Mangoldt coefficients at the level of ideals and
deduce the divisor-type majorant $|f|\le d_\kappa$. We then introduce the
$y$-friable/$y$-rough decomposition $f=s*\ell$ and prove the combinatorial identity on
which the method rests. After a short-interval bound for the ideal von Mangoldt
function and the resulting mean-square estimate, we truncate Perron's formula in
Proposition~\ref{prop:truncated}. This leads to the integral-form bound in
Proposition~\ref{prop:ideal-halasz-integral}, from which
Theorem~\ref{thm:ideal-halasz-M} follows.

The results of this section correspond to those of~\cite{GS} as follows.
Lemma~\ref{lem:Lambda-convolution} and Proposition~\ref{prop:lambda_implication} are the
ideal analogues of \cite[(1.2.4) and Exercise~1.2.1(iv)]{GS}, and
Proposition~\ref{prop:GS-mean-value} is the analogue of \cite[Proposition~1.2.2]{GS}.
Lemmas~\ref{lem:refined_integral_identity} and~\ref{lem:error_terms} correspond to
\cite[Lemmas~2.4.4 and~2.4.5]{GS}, Proposition~\ref{prop:truncated} to
\cite[Proposition~2.4.6]{GS}, and Lemma~\ref{lem:MeanSquareGaussian} to
\cite[Lemma~2.5.3]{GS}. Lemma~\ref{lem:S-ratio},
Proposition~\ref{prop:step3-mainbound} and Proposition~\ref{prop:ideal-halasz-integral}
make up the proof of \cite[Theorem~2.5.1]{GS} in \cite[\S2.5.2]{GS}, and
Theorem~\ref{thm:ideal-halasz-M} is the analogue of \cite[Corollary~2.5.2]{GS}.

In~\cite{GHS19}, the corresponding results over $\N$ are Theorem~1.1 and Corollary~1.2,
proved in \cite[\S2]{GHS19}. There Lemmas~2.2 and~2.4 correspond to our
Lemmas~\ref{lem:refined_integral_identity} and~\ref{lem:error_terms}, Proposition~2.1
together with Lemma~2.5 to Proposition~\ref{prop:truncated}, and Lemma~2.6 to
Lemma~\ref{lem:MeanSquareGaussian}. The proof of Theorem~1.1 in \cite[\S2.3]{GHS19}
corresponds to Lemma~\ref{lem:S-ratio} and Propositions~\ref{prop:step3-mainbound}
and~\ref{prop:ideal-halasz-integral}.

The main new ingredient is Lemma~\ref{lem:ideal-vm-short-interval}. It replaces the
Brun--Titchmarsh theorem in the proof of \cite[Lemma~2.5.3]{GS}. In
Proposition~\ref{prop:truncated} it also replaces the short-interval bound for $d_\kappa$
in \cite[Exercise~2.5.3]{GS}, for which \cite{GHS19} use Shiu's theorem
\cite[Lemma~2.3]{GHS19}.

Throughout this section and Appendix~\ref{app:aux}, $C_0\ge 1$ denotes an absolute
constant such that
\begin{equation}\label{eq:C0}
\psi_{\Q(i)}(u)\le C_0\,u\qquad(u\ge 1),
\end{equation}
which exists by Lemma~\ref{lem:PIT}. For a function $g:\mathcal I\to\C$ and $u\in\R$ we
write
\[
g_u(\mathfrak a):=g(\mathfrak a)(\Na\mathfrak a)^{-u}.
\]
Since the norm is completely multiplicative, $(g*h)_u=g_u*h_u$.

\subsection{Associated von Mangoldt coefficients}

Let $f:\mathcal I\to\C$ be multiplicative. For each prime ideal $\mathfrak p$ let
\[
F_{\mathfrak p}(z):=\sum_{m\ge0}f(\mathfrak p^m)z^m,
\]
so that the Euler factor of $F(s)=\sum_{\mathfrak a}f(\mathfrak a)(\Na\mathfrak a)^{-s}$ at
$\mathfrak p$ is $F_{\mathfrak p}((\Na\mathfrak p)^{-s})$. We regard $F_{\mathfrak p}$ as a
power series in $z$. Since $F_{\mathfrak p}(0)=1$, the quotient $F'_{\mathfrak p}/F_{\mathfrak p}$
is again a power series in $z$.

\begin{definition}[Associated von Mangoldt coefficients]\label{def:Lambda_f_coeff}
Let $f:\mathcal I\to\C$ be multiplicative. For each prime ideal $\mathfrak p$ define
$\Lambda_f(\mathfrak p^k)$, $k\ge1$, by the identity of power series
\[
z\,\frac{F'_{\mathfrak p}(z)}{F_{\mathfrak p}(z)}
=\sum_{k\ge1}\frac{\Lambda_f(\mathfrak p^k)}{\log\Na\mathfrak p}\,z^k ,
\]
and set $\Lambda_f(\mathfrak a)=0$ if $\mathfrak a$ is not of the form $\mathfrak p^k$ with $k\ge1$.
We call $\Lambda_f$ the associated von Mangoldt coefficients of $f$.
\end{definition}

Multiplying the defining identity by $F_{\mathfrak p}(z)$ and comparing coefficients of
$z^m$, we obtain the recursion
\begin{equation}\label{eq:lambda-f-recursion}
m\,f(\mathfrak p^m)\log\Na\mathfrak p
=\sum_{j=1}^{m} f(\mathfrak p^{m-j})\,\Lambda_f(\mathfrak p^{j})
\qquad(m\ge1),
\end{equation}
which determines $\Lambda_f(\mathfrak p^m)$ recursively. In particular
$\Lambda_f(\mathfrak p)=f(\mathfrak p)\log\Na\mathfrak p$, and if $f$ is completely multiplicative
then $\Lambda_f(\mathfrak p^k)=f(\mathfrak p)^k\log\Na\mathfrak p$.

\begin{lemma}\label{lem:Lambda-convolution}
Let $f:\mathcal I\to\C$ be multiplicative. Then
\[
f(\mathfrak a)\log\Na\mathfrak a=(\Lambda_f*f)(\mathfrak a)
\qquad(\mathfrak a\in\mathcal I).
\]
\end{lemma}

\begin{proof}
This is the ideal analogue of \cite[(1.2.4)]{GS}. Write
$\mathfrak a=\prod_{\mathfrak p}\mathfrak p^{v_{\mathfrak p}}$, and for $\mathfrak p\mid\mathfrak a$ let
$\mathfrak a_{\mathfrak p}':=\mathfrak a\,\mathfrak p^{-v_{\mathfrak p}}$. Since $\Lambda_f$ is supported on
prime powers and $f$ is multiplicative,
\[
(\Lambda_f*f)(\mathfrak a)
=\sum_{\mathfrak p\mid\mathfrak a}f(\mathfrak a_{\mathfrak p}')
\sum_{j=1}^{v_{\mathfrak p}}\Lambda_f(\mathfrak p^j)f(\mathfrak p^{v_{\mathfrak p}-j}).
\]
By \eqref{eq:lambda-f-recursion}, the inner sum equals
$v_{\mathfrak p}f(\mathfrak p^{v_{\mathfrak p}})\log\Na\mathfrak p$. Hence the right-hand side is
$f(\mathfrak a)\sum_{\mathfrak p\mid\mathfrak a}v_{\mathfrak p}\log\Na\mathfrak p=f(\mathfrak a)\log\Na\mathfrak a$.
\end{proof}

\begin{proposition}\label{prop:lambda_implication}
Let $f$ be multiplicative on ideals. If its associated von Mangoldt coefficients satisfy
$|\Lambda_f(\mathfrak a)|\le \kappa\Lambda(\mathfrak a)$ for some $\kappa\ge 1$, then
$|f(\mathfrak a)|\le d_{\kappa}(\mathfrak a)$ for all ideals $\mathfrak a$.
\end{proposition}

\begin{proof}
This is essentially \cite[Exercise~1.2.1(iv)]{GS}; we include the proof for completeness.
Define $d_\kappa$ on prime powers by the generating function
\[
\sum_{m\ge 0} d_\kappa(\mathfrak p^m) z^m = (1-z)^{-\kappa},
\]
so $d_\kappa$ is multiplicative and
$\zeta_{\Q(i)}(s)^{\kappa}=\sum_{\mathfrak a}d_\kappa(\mathfrak a)/(\Na\mathfrak a)^s$ for $\Re(s)>1$.
Differentiating $\log(1-z)^{-\kappa}$ shows that
\begin{equation}\label{eq:dkappa-recursion}
m\, d_\kappa(\mathfrak p^m)=\kappa \sum_{j=1}^{m} d_\kappa(\mathfrak p^{m-j})
\qquad(m\ge 1).
\end{equation}
We prove by induction on $m$ that $|f(\mathfrak p^m)|\le d_\kappa(\mathfrak p^m)$. For $m=0$
both sides equal $1$. Assume the claim holds for all smaller exponents. Using
\eqref{eq:lambda-f-recursion} and $|\Lambda_f(\mathfrak p^j)|\le \kappa\log\Na\mathfrak p$, we get
\[
m\,|f(\mathfrak p^m)|
\le \sum_{j=1}^{m} |f(\mathfrak p^{m-j})|\,\frac{|\Lambda_f(\mathfrak p^j)|}{\log\Na\mathfrak p}
\le \kappa \sum_{j=1}^{m} d_\kappa(\mathfrak p^{m-j}).
\]
By \eqref{eq:dkappa-recursion} the right-hand side equals $m\,d_\kappa(\mathfrak p^m)$. The
general case follows by multiplicativity.
\end{proof}

\begin{remark}\label{rem:Lambda-Dirichlet}
Assume that $|\Lambda_f(\mathfrak a)|\le\kappa\Lambda(\mathfrak a)$ for all $\mathfrak a$. By
Proposition~\ref{prop:lambda_implication}, the series
$F(s)=\sum_{\mathfrak a}f(\mathfrak a)(\Na\mathfrak a)^{-s}$ converges absolutely for $\Re(s)>1$.
We claim that
\begin{equation}\label{eq:logF}
F(s)=\exp\Bigl(\sum_{\mathfrak p}\sum_{k\ge1}\frac{\Lambda_f(\mathfrak p^k)}{k\log\Na\mathfrak p}
(\Na\mathfrak p)^{-ks}\Bigr)\qquad(\Re(s)>1),
\end{equation}
with the double series converging absolutely. In particular $F(s)\neq0$ and
\[
-\frac{F'}{F}(s)=\sum_{\mathfrak a\in\mathcal I}\frac{\Lambda_f(\mathfrak a)}{(\Na\mathfrak a)^s}
\qquad(\Re(s)>1),
\]
so Definition~\ref{def:Lambda_f_coeff} agrees with the description of $\Lambda_f$ in
Theorem~\ref{thm:ideal-halasz-M}.

To prove \eqref{eq:logF}, put
$G_{\mathfrak p}(z):=\sum_{k\ge1}\Lambda_f(\mathfrak p^k)z^k/(k\log\Na\mathfrak p)$. Then
$zG_{\mathfrak p}'=zF_{\mathfrak p}'/F_{\mathfrak p}$ as power series, so $\exp(G_{\mathfrak p})$ and
$F_{\mathfrak p}$ satisfy the same first-order differential equation with the same constant term.
Hence $F_{\mathfrak p}=\exp(G_{\mathfrak p})$ as power series. Since $|\Lambda_f(\mathfrak p^k)|\le\kappa\log\Na\mathfrak p$
and $|f(\mathfrak p^m)|\le d_\kappa(\mathfrak p^m)\ll_\kappa(m+1)^\kappa$, both series converge for
$|z|<1$, and therefore $F_{\mathfrak p}(z)=\exp(G_{\mathfrak p}(z))$ there. Taking
$z=(\Na\mathfrak p)^{-s}$ in the Euler product $F(s)=\prod_{\mathfrak p}F_{\mathfrak p}((\Na\mathfrak p)^{-s})$
gives \eqref{eq:logF}. The double series converges absolutely because
$\sum_{\mathfrak p}\sum_{k\ge1}\kappa(\Na\mathfrak p)^{-k\sigma}/k<\infty$ for $\sigma>1$, and the
formula for $-F'/F$ follows by termwise differentiation.
\end{remark}

\subsection{The friable--rough decomposition}

For the rest of this section, $\kappa\ge1$ and $f:\mathcal I\to\C$ is a multiplicative
function whose associated von Mangoldt coefficients satisfy
\[
|\Lambda_f(\mathfrak a)|\le\kappa\Lambda(\mathfrak a)\qquad(\mathfrak a\in\mathcal I).
\]
By Proposition~\ref{prop:lambda_implication}, $|f(\mathfrak a)|\le d_\kappa(\mathfrak a)$.

\begin{definition}\label{def:smooth-rough}
Let $y\ge 2$.
\begin{enumerate}
\item An ideal $\mathfrak a\in\mathcal I$ is $y$-friable if every prime ideal $\mathfrak p\mid \mathfrak a$ satisfies $\Na\mathfrak p\le y$.
\item An ideal $\mathfrak a\in\mathcal I$ is $y$-rough if every prime ideal $\mathfrak p\mid \mathfrak a$ satisfies $\Na\mathfrak p> y$.
\end{enumerate}
Given multiplicative $f$, define
\[
 s(\mathfrak a):=\begin{cases} f(\mathfrak a)&\text{if $\mathfrak a$ is $y$-friable,}\\ 0&\text{otherwise,}\end{cases}
\qquad
 \ell(\mathfrak a):=\begin{cases} f(\mathfrak a)&\text{if $\mathfrak a$ is $y$-rough,}\\ 0&\text{otherwise.}\end{cases}
\]
For $\Re(s)>1$, set
\[
F(s)=\sum_{\mathfrak a\in\mathcal I}\frac{f(\mathfrak a)}{(\Na\mathfrak a)^s},\quad
S(s)=\sum_{\mathfrak a\in\mathcal I}\frac{s(\mathfrak a)}{(\Na\mathfrak a)^s},\quad
L(s)=\sum_{\mathfrak a\in\mathcal I}\frac{\ell(\mathfrak a)}{(\Na\mathfrak a)^s}.
\]
\end{definition}

The functions $s$ and $\ell$ are multiplicative. For each prime ideal $\mathfrak p$, the power
series $\sum_{m\ge0}s(\mathfrak p^m)z^m$ equals $F_{\mathfrak p}(z)$ if $\Na\mathfrak p\le y$ and
$1$ otherwise, while $\sum_{m\ge0}\ell(\mathfrak p^m)z^m$ equals $F_{\mathfrak p}(z)$ if
$\Na\mathfrak p>y$ and $1$ otherwise. Definition~\ref{def:Lambda_f_coeff} therefore gives
\begin{equation}\label{eq:Lambda-ell}
\Lambda_\ell(\mathfrak p^k)=\Lambda_f(\mathfrak p^k)\mathbf 1_{\Na\mathfrak p>y},
\qquad
\Lambda_s(\mathfrak p^k)=\Lambda_f(\mathfrak p^k)\mathbf 1_{\Na\mathfrak p\le y}
\qquad(k\ge1).
\end{equation}
In particular $|\Lambda_\ell|,|\Lambda_s|\le\kappa\Lambda$ and $|s|,|\ell|\le d_\kappa$. Also
$\Lambda_\ell(\mathfrak c)\ne0$ implies $\Na\mathfrak c>y$, and by Remark~\ref{rem:Lambda-Dirichlet}
we have $-L'/L(s)=\sum_{\mathfrak a}\Lambda_\ell(\mathfrak a)(\Na\mathfrak a)^{-s}$ for $\Re(s)>1$.
Finally, $S(s)$ converges absolutely for $\Re(s)>0$, since
\begin{equation}\label{eq:S-abs}
\sum_{\mathfrak a}\frac{|s(\mathfrak a)|}{(\Na\mathfrak a)^{\sigma}}
\le\prod_{\Na\mathfrak p\le y}\Bigl(1-\frac1{(\Na\mathfrak p)^{\sigma}}\Bigr)^{-\kappa}<\infty
\qquad(\sigma>0).
\end{equation}

\begin{lemma}\label{lem:convolution}
With $s,\ell$ as in Definition~\ref{def:smooth-rough}, one has
\[
f=s*\ell
\]
on $\mathcal I$, where $*$ denotes the Dirichlet convolution from
Definition~\ref{def:ideal-convolution}.
\end{lemma}

\begin{proof}
Write $\mathfrak a=\prod_{\mathfrak p}\mathfrak p^{v_{\mathfrak p}(\mathfrak a)}$.
Let $\mathfrak a_s:=\prod_{\Na\mathfrak p\le y}\mathfrak p^{v_{\mathfrak p}(\mathfrak a)}$ and
$\mathfrak a_\ell:=\prod_{\Na\mathfrak p> y}\mathfrak p^{v_{\mathfrak p}(\mathfrak a)}$.
Then $\mathfrak a=\mathfrak a_s\mathfrak a_\ell$ with $(\mathfrak a_s,\mathfrak a_\ell)=\Z[i]$.
In the divisor sum, the only $\mathfrak d$ for which both $s(\mathfrak d)$ and $\ell(\mathfrak a/\mathfrak d)$ are nonzero
is $\mathfrak d=\mathfrak a_s$.
Therefore the divisor sum equals $s(\mathfrak a_s)\,\ell(\mathfrak a_\ell)=f(\mathfrak a_s)f(\mathfrak a_\ell)=f(\mathfrak a)$.
\end{proof}

\begin{corollary}\label{cor:factor}
For $\Re(s)>1$ we have $F(s)=S(s)L(s)$.
\end{corollary}

\subsection{The basic identity}

For $\alpha,\beta\ge0$ define
\[
a_{\alpha,\beta}:=s*(\Lambda_\ell)_{\alpha}*(\Lambda_\ell)_{\alpha+\beta}*\ell_{\alpha+\beta},
\]
that is,
\[
a_{\alpha,\beta}(\mathfrak a)=\sum_{\mathfrak b\mathfrak c\mathfrak c'\mathfrak d=\mathfrak a}
s(\mathfrak b)\,\frac{\Lambda_\ell(\mathfrak c)}{(\Na\mathfrak c)^{\alpha}}\,
\frac{\Lambda_\ell(\mathfrak c')}{(\Na\mathfrak c')^{\alpha+\beta}}\,
\frac{\ell(\mathfrak d)}{(\Na\mathfrak d)^{\alpha+\beta}} .
\]

\begin{lemma}\label{lem:refined_integral_identity}
Let $y\ge2$, $\eta>0$ and $x\ge1$. Then
\begin{align}
\sum_{\Na\mathfrak a\le x}f(\mathfrak a)
&=\int_0^\eta\!\!\int_0^\eta\sum_{\Na\mathfrak a\le x}a_{\alpha,\beta}(\mathfrak a)\,d\beta\,d\alpha
+\sum_{\Na(\mathfrak b\mathfrak d)\le x}s(\mathfrak b)\,\frac{\ell(\mathfrak d)}{(\Na\mathfrak d)^{\eta}}
\label{eq:refined-identity}\\
&\quad+\int_0^\eta\sum_{\Na(\mathfrak b\mathfrak c\mathfrak d)\le x}s(\mathfrak b)\,
\frac{\Lambda_\ell(\mathfrak c)}{(\Na\mathfrak c)^{\alpha}}\,
\frac{\ell(\mathfrak d)}{(\Na\mathfrak d)^{\alpha+\eta}}\,d\alpha. \notag
\end{align}
Moreover, for $\alpha,\beta\ge0$,
\[
\sum_{\mathfrak a\in\mathcal I}\frac{a_{\alpha,\beta}(\mathfrak a)}{(\Na\mathfrak a)^s}
=S(s)\,L(s+\alpha+\beta)\,\frac{L'}{L}(s+\alpha)\,\frac{L'}{L}(s+\alpha+\beta)
\qquad(\Re(s)>1),
\]
the series converging absolutely. Consequently, if $x$ is not the norm of an ideal and
$c>1$, the first term on the right of \eqref{eq:refined-identity} equals
\[
\int_{0}^{\eta}\!\int_{0}^{\eta}\!\frac{1}{2\pi i}\int_{c-i\infty}^{c+i\infty}
S(s)\,L(s+\alpha+\beta)\,\frac{L'}{L}(s+\alpha)\,\frac{L'}{L}(s+\alpha+\beta)\,\frac{x^{s}}{s}\,ds\,d\beta\,d\alpha .
\]
\end{lemma}

\begin{proof}
This is the ideal analogue of \cite[Lemma~2.4.4]{GS}; compare also
\cite[Lemma~2.2]{GHS19}, where $\beta$ ranges over $[0,2\eta]$. In
\cite{GS} the identity is obtained by integrating the Dirichlet series in $\alpha$ and
$\beta$ and then applying Perron's formula. Here we work directly with the coefficients,
which avoids interchanging the integrals.

Fix $u\ge0$. For every ideal $\mathfrak d$,
\[
\ell_u(\mathfrak d)-\ell_{u+\eta}(\mathfrak d)
=\int_0^\eta\ell(\mathfrak d)\log\Na\mathfrak d\,(\Na\mathfrak d)^{-u-\beta}\,d\beta .
\]
By Lemma~\ref{lem:Lambda-convolution} applied to $\ell$, we have
$\ell\log\Na=\Lambda_\ell*\ell$, and hence
\begin{equation}\label{eq:ell-step}
\ell_u=\ell_{u+\eta}+\int_0^\eta(\Lambda_\ell)_{u+\beta}*\ell_{u+\beta}\,d\beta
\qquad(u\ge0)
\end{equation}
pointwise on $\mathcal I$. We apply \eqref{eq:ell-step} with $u=0$, renaming the variable
of integration $\alpha$, and then apply it again with $u=\alpha$ to the factor $\ell_\alpha$.
Since all convolutions involved are finite sums, this gives
\[
\ell=\ell_\eta+\int_0^\eta(\Lambda_\ell)_\alpha*\ell_{\alpha+\eta}\,d\alpha
+\int_0^\eta\!\!\int_0^\eta(\Lambda_\ell)_\alpha*(\Lambda_\ell)_{\alpha+\beta}*\ell_{\alpha+\beta}
\,d\beta\,d\alpha .
\]
Convolving with $s$, using $f=s*\ell$ from Lemma~\ref{lem:convolution}, and summing over
$\Na\mathfrak a\le x$ gives \eqref{eq:refined-identity}.

By Remark~\ref{rem:Lambda-Dirichlet} applied to $\ell$, the Dirichlet series of
$(\Lambda_\ell)_u$ is $-\frac{L'}{L}(s+u)$, and that of $\ell_u$ is $L(s+u)$. Both converge
absolutely for $\Re(s)>1$ and $u\ge0$, and so does $S(s)$ by \eqref{eq:S-abs}. This gives
the formula for the Dirichlet series of $a_{\alpha,\beta}$. The last assertion then follows
from Theorem~\ref{thm:perron}.
\end{proof}

The next lemma handles the error terms.

\begin{lemma}\label{lem:error_terms}
Let $x\ge y\ge e^6$ and set $\eta:=1/\log y$. Then
\begin{align*}
&\sum_{\Na(\mathfrak b\mathfrak d) \le x} |s(\mathfrak b)| \, \frac{|\ell(\mathfrak d)|}{(\Na\mathfrak d)^\eta}
+\int_0^\eta \sum_{\Na(\mathfrak b\mathfrak c\mathfrak d) \le x} |s(\mathfrak b)|\, \frac{|\Lambda_{\ell}(\mathfrak c)|}{(\Na\mathfrak c)^\alpha} \, \frac{|\ell(\mathfrak d)|}{(\Na\mathfrak d)^{\eta+\alpha}} \,d\alpha
\ \ll_{\kappa}\ \frac{x}{\log x}(\log y)^\kappa.
\end{align*}
\end{lemma}

\begin{proof}
We follow the proof of \cite[Lemma~2.4.5]{GS}; compare also \cite[Lemma~2.4]{GHS19}.
The only change is that we apply
Proposition~\ref{prop:GS-mean-value} directly to the majorant $g$ below, where
\cite{GS} appeals to \cite[Corollary~1.2.3]{GS} and \cite{GHS19} to Shiu's theorem.

We begin with the first sum. Define a nonnegative multiplicative function $g$ by
\[
g(\mathfrak p^m):=
\begin{cases}
d_\kappa(\mathfrak p^m), & \Na\mathfrak p\le y,\\
d_\kappa(\mathfrak p^m)(\Na\mathfrak p)^{-m\eta}, & \Na\mathfrak p>y.
\end{cases}
\]
For each $\mathfrak a$, the only pair $(\mathfrak b,\mathfrak d)$ with $\mathfrak b\mathfrak d=\mathfrak a$ contributing
to the first sum is $(\mathfrak a_s,\mathfrak a_\ell)$, in the notation of the proof of
Lemma~\ref{lem:convolution}. Since $|s|,|\ell|\le d_\kappa$, it follows that
\[
\sum_{\Na(\mathfrak b\mathfrak d)\le x}|s(\mathfrak b)|\,\frac{|\ell(\mathfrak d)|}{(\Na\mathfrak d)^\eta}
\le\sum_{\Na\mathfrak a\le x}g(\mathfrak a).
\]
For each prime ideal $\mathfrak p$, the power series $\sum_{m\ge0}g(\mathfrak p^m)z^m$ equals
$(1-z)^{-\kappa}$ if $\Na\mathfrak p\le y$, and $(1-(\Na\mathfrak p)^{-\eta}z)^{-\kappa}$ if
$\Na\mathfrak p>y$. Hence
$\Lambda_g(\mathfrak p^k)$ equals $\kappa\log\Na\mathfrak p$ or
$\kappa(\Na\mathfrak p)^{-k\eta}\log\Na\mathfrak p$ respectively. Thus $0\le\Lambda_g\le\kappa\Lambda$,
and by \eqref{eq:C0},
\[
\sum_{\Na\mathfrak m\le z}\Lambda_g(\mathfrak m)\le\kappa C_0z\qquad(z\ge1).
\]
Proposition~\ref{prop:GS-mean-value} therefore gives
\begin{align*}
\sum_{\Na\mathfrak a\le x}g(\mathfrak a)
&\le\frac{(\kappa C_0+1)x}{\log x}\sum_{\mathfrak a}\frac{g(\mathfrak a)}{\Na\mathfrak a}\\
&=\frac{(\kappa C_0+1)x}{\log x}
\prod_{\Na\mathfrak p\le y}\Bigl(1-\frac1{\Na\mathfrak p}\Bigr)^{-\kappa}
\prod_{\Na\mathfrak p>y}\Bigl(1-\frac1{(\Na\mathfrak p)^{1+\eta}}\Bigr)^{-\kappa}.
\end{align*}
By Lemma~\ref{lem:Mertens} and Lemma~\ref{lem:euler-products}(b) with $z=y$, this is
$\ll_\kappa x(\log y)^\kappa/\log x$.

We now turn to the $\alpha$-integral. Fix $\alpha\in[0,\eta]$, and note that $\eta\le1/6$.
Using $|\Lambda_\ell|\le\kappa\Lambda$ and summing first over $\mathfrak c$ by
Lemma~\ref{lem:weighted-Lambda}(iii), we obtain
\begin{align*}
\sum_{\Na(\mathfrak b\mathfrak c\mathfrak d)\le x}|s(\mathfrak b)|\,
\frac{|\Lambda_\ell(\mathfrak c)|}{(\Na\mathfrak c)^\alpha}\,
\frac{|\ell(\mathfrak d)|}{(\Na\mathfrak d)^{\eta+\alpha}}
&\ll\kappa\sum_{\Na(\mathfrak b\mathfrak d)\le x}\frac{|s(\mathfrak b)|\,|\ell(\mathfrak d)|}{(\Na\mathfrak d)^{\eta+\alpha}}
\Bigl(\frac{x}{\Na(\mathfrak b\mathfrak d)}\Bigr)^{1-\alpha}\\
&\le\kappa\,x^{1-\alpha}\sum_{\mathfrak b}\frac{|s(\mathfrak b)|}{(\Na\mathfrak b)^{1-\alpha}}
\sum_{\mathfrak d}\frac{|\ell(\mathfrak d)|}{(\Na\mathfrak d)^{1+\eta}} .
\end{align*}
Since $|s|,|\ell|\le d_\kappa$, the two sums on the right are bounded by Euler products:
\begin{align*}
\sum_{\mathfrak b}\frac{|s(\mathfrak b)|}{(\Na\mathfrak b)^{1-\alpha}}
&\le\prod_{\Na\mathfrak p\le y}\Bigl(1-\frac1{(\Na\mathfrak p)^{1-\alpha}}\Bigr)^{-\kappa},\\
\sum_{\mathfrak d}\frac{|\ell(\mathfrak d)|}{(\Na\mathfrak d)^{1+\eta}}
&\le\prod_{\Na\mathfrak p>y}\Bigl(1-\frac1{(\Na\mathfrak p)^{1+\eta}}\Bigr)^{-\kappa}.
\end{align*}
By Lemma~\ref{lem:euler-products}(a),(b), these products are $\ll_\kappa(\log y)^\kappa$ and
$\ll_\kappa1$ respectively. Finally,
\[
\int_0^\eta x^{1-\alpha}\,d\alpha\le x\int_0^\infty e^{-\alpha\log x}\,d\alpha=\frac{x}{\log x},
\]
and combining the two bounds proves the lemma.
\end{proof}

\subsection{An $L^2$-estimate}

We next prove the short-interval bound for the ideal von Mangoldt function and the
$L^2$-estimate it implies. The short-interval bound is also used in the truncation step,
Proposition~\ref{prop:truncated}.

\begin{lemma}[Short-interval bound for the ideal von Mangoldt function]
\label{lem:ideal-vm-short-interval}
Let $T\ge1$ and $M\ge T^2/4$. Then
\[
\sum_{M e^{-1/T}\le \Na\mathfrak a\le M e^{1/T}} \Lambda(\mathfrak a)
\ll \frac{M}{T},
\]
with an absolute implied constant.
\end{lemma}

\begin{proof}
Over $\N$, the corresponding bound is the Brun--Titchmarsh step in the proofs of
\cite[Lemma~2.5.3]{GS} and \cite[Lemma~2.6]{GHS19}. Over $\Z[i]$ we reduce to the Brun--Titchmarsh theorem for primes
$p\equiv1\pmod 4$. Let $U_\pm:=Me^{\pm1/T}$ and
$H:=U_+-U_-=2M\sinh(1/T)$. Since
$a\le\sinh a\le(\sinh 1)a$ for $0\le a\le1$,
\begin{equation}\label{eq:H-size}
\frac{2M}{T}\le H\le\frac{4M}{T}.
\end{equation}
We must show
\begin{equation}\label{eq:shortinterval-Lambda}
\sum_{U_-\le \Na\mathfrak a\le U_+}\Lambda(\mathfrak a)\ll H.
\end{equation}
If $T<6$, the sum is at most $\psi_{\Q(i)}(U_+)\le eC_0M\ll H$, by \eqref{eq:C0} and
\eqref{eq:H-size}. We may therefore assume $T\ge6$. Then $U_+/U_-=e^{2/T}<2$ and
$U_-\ge T^2/(4e)\ge3$, and \eqref{eq:H-size} gives
$3\le T/2\le H\le 4M/T\le U_-$. Moreover, $T\le2\sqrt M$ gives $2M/T\ge\sqrt M$, so
\begin{equation}\label{eq:sqrtM}
\sqrt M\le H\qquad\text{and}\qquad\log H\ge\tfrac12\log M .
\end{equation}

Since $\Lambda(\mathfrak a)$ is supported on prime powers, write
\[
\sum_{U_-\le \Na\mathfrak a\le U_+}\Lambda(\mathfrak a)=S_1+S_{\ge 2},
\]
where
\[
S_1:=\sum_{\substack{\mathfrak p\ \mathrm{prime}\\ U_-\le \Na\mathfrak p\le U_+}}\log \Na\mathfrak p,
\qquad
S_{\ge 2}:=\sum_{\substack{\mathfrak p\ \mathrm{prime},\,k\ge 2\\ U_-\le (\Na\mathfrak p)^k\le U_+}}\log \Na\mathfrak p.
\]

Since $U_+/U_-<2$, the interval $[U_-,U_+]$ contains at most one power $(\Na\mathfrak p)^k$
with $k\ge2$ for each fixed prime ideal $\mathfrak p$. Also $\Na\mathfrak p\le U_+^{1/2}$ for every
term in $S_{\ge2}$. Therefore, by \eqref{eq:C0} and \eqref{eq:sqrtM},
\[
S_{\ge 2}\le\psi_{\Q(i)}\bigl(U_+^{1/2}\bigr)\le C_0\sqrt{eM}\ll H.
\]

For $S_1$, note first that $U_-\ge3$, so the ramified prime ideal above $2$ does not occur.
The inert prime ideals have norm $p^2$ with $p\le U_+^{1/2}$, so their contribution is
$\ll\sqrt M\ll H$. Each rational prime $p\equiv1\pmod4$ lies below exactly two prime ideals
of norm $p$. A possible prime at the endpoint $p=U_-$ contributes at most
$2\log U_-\ll\log M\ll H$. Hence
\[
S_1\le 2\!\!\sum_{\substack{U_-< p\le U_+\\ p\equiv 1\ (\mathrm{mod}\ 4)}}\!\!\log p
\;+\;O(H).
\]
Since $U_-\ge3$ and $2\le H\le U_-$, Lemma~\ref{lem:BT_mod4} applies to the interval
$(U_-,U_-+H]=(U_-,U_+]$ and gives
\[
\sum_{\substack{U_-< p\le U_+\\ p\equiv 1\ (\mathrm{mod}\ 4)}}\log p
\ll \frac{H\log U_+}{\log (2H)}\le\frac{H(\log M+1)}{\tfrac12\log M}\ll H,
\]
by \eqref{eq:sqrtM}, since $M\ge T^2/4\ge9$. Therefore $S_1\ll H$, and
\eqref{eq:shortinterval-Lambda} follows.
\end{proof}

\begin{lemma}\label{lem:MeanSquareGaussian}
Let $T\ge 1$ and $x\ge T^2$.
For any complex numbers $\{c(\mathfrak a)\}$ indexed by nonzero ideals $\mathfrak a\subset \Z[i]$,
\[
\int_{-T}^{T}\Bigl|\sum_{T^2\le \Na\mathfrak a\le x} c(\mathfrak a)\,\Lambda(\mathfrak a)\,(\Na\mathfrak a)^{-it}\Bigr|^2\,dt
\ \ll\
\sum_{T^2\le \Na\mathfrak a\le x} \Na\mathfrak a\,|c(\mathfrak a)|^2\,\Lambda(\mathfrak a),
\]
with an absolute implied constant.
\end{lemma}

\begin{proof}
We follow the proof of \cite[Lemma~2.5.3]{GS} (see also \cite[Lemma~2.6]{GHS19}), with
Lemma~\ref{lem:ideal-vm-short-interval} in place of the Brun--Titchmarsh theorem. Denote the Dirichlet polynomial on the left by $A(t)$, and put
\[
\Phi(u):=c_1\Bigl(\frac{\sin(u/2)}{u/2}\Bigr)^2,
\qquad
c_1:=\Bigl(\frac{1/2}{\sin(1/2)}\Bigr)^2 .
\]
Since $v\mapsto\sin v/v$ is positive and decreasing on $[0,1/2]$, we have $\Phi(u)\ge1$ for
$|u|\le1$. Moreover $\Phi\ge0$ is integrable, and its Fourier transform
\[
\widehat\Phi(\xi):=\int_{\R}\Phi(u)e^{-iu\xi}\,du=2\pi c_1\,(1-|\xi|)_+
\]
is even, nonnegative, bounded by $2\pi c_1$, and supported in $[-1,1]$. Expanding the
square, we obtain
\[
\int_{-T}^T|A(t)|^2\,dt\le\int_{\R}|A(t)|^2\Phi(t/T)\,dt
=\sum_{\mathfrak a,\mathfrak b}c(\mathfrak a)\overline{c(\mathfrak b)}\Lambda(\mathfrak a)\Lambda(\mathfrak b)
\,T\,\widehat\Phi\Bigl(T\log\frac{\Na\mathfrak a}{\Na\mathfrak b}\Bigr).
\]
Using $|c(\mathfrak a)c(\mathfrak b)|\le\frac12(|c(\mathfrak a)|^2+|c(\mathfrak b)|^2)$ and the symmetry of
$\widehat\Phi$, the right-hand side is at most
\[
\sum_{\mathfrak a}|c(\mathfrak a)|^2\Lambda(\mathfrak a)
\sum_{\mathfrak b}\Lambda(\mathfrak b)\,T\,\widehat\Phi\Bigl(T\log\frac{\Na\mathfrak a}{\Na\mathfrak b}\Bigr).
\]
Fix $\mathfrak a$ and write $M:=\Na\mathfrak a\ge T^2$. The summand in $\mathfrak b$ vanishes unless
$\Na\mathfrak b\in[Me^{-1/T},Me^{1/T}]$, and it is at most $2\pi c_1T\Lambda(\mathfrak b)$. By
Lemma~\ref{lem:ideal-vm-short-interval}, the inner sum is therefore $\ll T\cdot M/T=M$,
and the lemma follows.
\end{proof}

\subsection{Reduction to a truncated integral}

\begin{proposition}\label{prop:truncated}
Let $x\ge3$ and $T\ge1$, and let $y$ satisfy $x>y\ge \max(T^2,e^6)$. Put
$\eta:=1/\log y$ and $c_0:=1+1/\log x$, and define the truncated log-derivative polynomial
\[
P(z):=\sum_{y<\Na\mathfrak c\le x/y}\frac{\Lambda_{\ell}(\mathfrak c)}{(\Na\mathfrak c)^z}.
\]
Then
\begin{align*}
\sum_{\Na\mathfrak a \le x} f(\mathfrak a)
&=\int_{0}^{\eta}\!\int_{0}^{\eta}\!\frac{1}{2\pi i}\int_{c_{0}-iT}^{c_{0}+iT}
S\bigl(s-\alpha-\beta/2\bigr)\,L\bigl(s+\beta/2\bigr)\\
&\qquad\qquad\times P(s-\beta/2)\,P(s+\beta/2)
\,\frac{x^{s-\alpha-\beta/2}}{s-\alpha-\beta/2}\,ds\,d\beta\,d\alpha\\
&\qquad\qquad
+ O_\kappa\!\left(\frac{x}{\log x}(\log y)^\kappa\right)
+ O_\kappa\!\left(\frac{x(\log x)^{\kappa+2}}{T\log y}\right).
\end{align*}
\end{proposition}

\begin{proof}
We follow the proof of \cite[Proposition~2.4.6]{GS}; compare also
\cite[Proposition~2.1 and Lemma~2.5]{GHS19}. There, the terms with $\Na\mathfrak a$
close to $x$ are bounded using a short-interval estimate for $d_\kappa$
(\cite[Exercise~2.5.3]{GS}, or Shiu's theorem in \cite[Lemma~2.3]{GHS19}). Here we use Lemma~\ref{lem:ideal-vm-short-interval} instead. By Lemma~\ref{lem:refined_integral_identity},
the sum on the left equals the double integral over $\alpha,\beta\in[0,\eta]$ of
$\sum_{\Na\mathfrak a\le x}a_{\alpha,\beta}(\mathfrak a)$, plus two further terms. By
Lemma~\ref{lem:error_terms}, these two terms contribute $O_\kappa(x(\log y)^\kappa/\log x)$.
It remains to treat $\sum_{\Na\mathfrak a\le x}a_{\alpha,\beta}(\mathfrak a)$ for fixed
$\alpha,\beta\in[0,\eta]$.

We first note that the logarithmic derivatives can be replaced by the polynomial $P$
without introducing an error. Let
\[
a^P_{\alpha,\beta}(\mathfrak a):=\sum_{\substack{\mathfrak b\mathfrak c\mathfrak c'\mathfrak d=\mathfrak a\\
y<\Na\mathfrak c,\,\Na\mathfrak c'\le x/y}}
s(\mathfrak b)\,\frac{\Lambda_\ell(\mathfrak c)}{(\Na\mathfrak c)^{\alpha}}\,
\frac{\Lambda_\ell(\mathfrak c')}{(\Na\mathfrak c')^{\alpha+\beta}}\,
\frac{\ell(\mathfrak d)}{(\Na\mathfrak d)^{\alpha+\beta}} .
\]
Consider a term of $a_{\alpha,\beta}(\mathfrak a)$ with $\Na\mathfrak a\le x$ and
$\Lambda_\ell(\mathfrak c)\Lambda_\ell(\mathfrak c')\neq0$. By \eqref{eq:Lambda-ell},
$\Na\mathfrak c>y$ and $\Na\mathfrak c'>y$. Hence $\Na\mathfrak c\le x/\Na\mathfrak c'<x/y$, and likewise
$\Na\mathfrak c'<x/y$. Therefore
\begin{equation}\label{eq:aP-exact}
\sum_{\Na\mathfrak a\le x}a_{\alpha,\beta}(\mathfrak a)=\sum_{\Na\mathfrak a\le x}a^P_{\alpha,\beta}(\mathfrak a).
\end{equation}
The Dirichlet series of $a^P_{\alpha,\beta}$ is
\[
A^P_{\alpha,\beta}(w):=S(w)\,P(w+\alpha)\,P(w+\alpha+\beta)\,L(w+\alpha+\beta).
\]
It converges absolutely for $\Re(w)>\max(0,1-\alpha-\beta)$, by \eqref{eq:S-abs} and since
$P$ is a finite sum.

Now put $\sigma':=c_0-\alpha-\beta/2$. Then $\sigma'+\alpha+\beta=c_0+\beta/2>1$ and
$\sigma'\ge1-\tfrac32\eta>0$, so $A^P_{\alpha,\beta}$ converges absolutely on the line
$\Re(w)=\sigma'$. We apply Lemma~\ref{lem:perron-quant} with $\sigma=\sigma'$, which is
valid for all $x\ge2$ by Remark~\ref{rem:perron-boundary}, and substitute
$w=s-\alpha-\beta/2$. By \eqref{eq:aP-exact} this gives
\[
\sum_{\Na\mathfrak a\le x}a_{\alpha,\beta}(\mathfrak a)
=\frac1{2\pi i}\int_{c_0-iT}^{c_0+iT}A^P_{\alpha,\beta}(s-\alpha-\beta/2)
\frac{x^{s-\alpha-\beta/2}}{s-\alpha-\beta/2}\,ds+O(R_{\alpha,\beta}),
\]
where
\[
R_{\alpha,\beta}:=\sum_{\mathfrak a}|a^P_{\alpha,\beta}(\mathfrak a)|
\Bigl(\frac{x}{\Na\mathfrak a}\Bigr)^{\sigma'}
\min\Bigl\{1,\frac1{T|\log(x/\Na\mathfrak a)|}\Bigr\}.
\]
Since
\[
A^P_{\alpha,\beta}(s-\alpha-\beta/2)=S(s-\alpha-\beta/2)\,L(s+\beta/2)\,P(s-\beta/2)\,P(s+\beta/2),
\]
the integral is exactly the one in the statement. It therefore remains to show that
\begin{equation}\label{eq:R-bound}
R_{\alpha,\beta}\ll_\kappa\frac{x^{1-\alpha}(\log x)^{\kappa+2}}{T}
\qquad(0\le\alpha,\beta\le\eta).
\end{equation}
Indeed, \eqref{eq:R-bound} implies
\[
\int_0^\eta\!\!\int_0^\eta R_{\alpha,\beta}\,d\beta\,d\alpha
\ll_\kappa\frac{(\log x)^{\kappa+2}}{T}\cdot\eta\cdot\frac{x}{\log x}
\le\frac{x(\log x)^{\kappa+2}}{T\log y}.
\]

To prove \eqref{eq:R-bound}, note that $|a^P_{\alpha,\beta}|\le\kappa^2b$ by
\eqref{eq:Lambda-ell}, where
\[
b(\mathfrak a):=\sum_{\substack{\mathfrak b\mathfrak c\mathfrak c'\mathfrak d=\mathfrak a\\
y<\Na\mathfrak c,\,\Na\mathfrak c'\le x/y}}
|s(\mathfrak b)|\,\frac{\Lambda(\mathfrak c)}{(\Na\mathfrak c)^{\alpha}}\,
\frac{\Lambda(\mathfrak c')}{(\Na\mathfrak c')^{\alpha+\beta}}\,
\frac{|\ell(\mathfrak d)|}{(\Na\mathfrak d)^{\alpha+\beta}} .
\]
We use two bounds for $b$. The first is
\begin{equation}\label{eq:b-total}
x^{\sigma'}\sum_{\mathfrak a}\frac{b(\mathfrak a)}{(\Na\mathfrak a)^{\sigma'}}
\ll_\kappa x^{1-\alpha}(\log x)^{\kappa+2}.
\end{equation}
Since $\sigma'+\alpha=c_0-\beta/2$ and $\sigma'+\alpha+\beta=c_0+\beta/2$, the sum on the left
factors as
\begin{multline*}
\Bigl(\sum_{\mathfrak b}\frac{|s(\mathfrak b)|}{(\Na\mathfrak b)^{\sigma'}}\Bigr)
\Bigl(\sum_{y<\Na\mathfrak c\le x/y}\frac{\Lambda(\mathfrak c)}{(\Na\mathfrak c)^{c_0-\beta/2}}\Bigr)\\
\times\Bigl(\sum_{y<\Na\mathfrak c'\le x/y}\frac{\Lambda(\mathfrak c')}{(\Na\mathfrak c')^{c_0+\beta/2}}\Bigr)
\Bigl(\sum_{\mathfrak d}\frac{|\ell(\mathfrak d)|}{(\Na\mathfrak d)^{c_0+\beta/2}}\Bigr).
\end{multline*}
Since $|s|\le d_\kappa$ and $\sigma'=1-(\alpha+\beta/2-1/\log x)$ with
$\alpha+\beta/2\le\tfrac32\eta$, Lemma~\ref{lem:euler-products}(a) bounds the first factor
by $\ll_\kappa(\log y)^\kappa$. Bounding $(\Na\mathfrak c)^{-(c_0-\beta/2)}$ by
$(\Na\mathfrak c)^{-(1-\beta/2)}$, we see from Lemma~\ref{lem:weighted-Lambda}(ii) that the
second factor is $\ll(x/y)^{\beta/2}\log x$,
and Lemma~\ref{lem:weighted-Lambda}(i) bounds the third by $\ll\log x$. Finally, the fourth
factor is at most
\[
\prod_{\Na\mathfrak p>y}\Bigl(1-\frac1{(\Na\mathfrak p)^{c_0}}\Bigr)^{-\kappa}
\ll_\kappa\Bigl(\frac{\log x}{\log y}\Bigr)^\kappa,
\]
by Lemma~\ref{lem:euler-products}(c). Since $x^{\sigma'}=e\,x^{1-\alpha-\beta/2}$ and
$(x/y)^{\beta/2}\le x^{\beta/2}$, multiplying these bounds gives \eqref{eq:b-total}.

The second is that, for $1\le V\le T$,
\begin{equation}\label{eq:b-near}
\sum_{|\log(x/\Na\mathfrak a)|\le V/T}b(\mathfrak a)\ll_\kappa\frac VT\,x^{1-\alpha}(\log x)^{\kappa+1}.
\end{equation}
To prove \eqref{eq:b-near}, fix $\mathfrak b,\mathfrak c,\mathfrak d$
and write $M:=x/\Na(\mathfrak b\mathfrak c\mathfrak d)$. The condition $|\log(x/\Na\mathfrak a)|\le V/T$
then reads $\Na\mathfrak c'\in[Me^{-V/T},Me^{V/T}]$. If this range contains some $\mathfrak c'$
with $\Na\mathfrak c'>y\ge T^2$, then $M\ge T^2e^{-V/T}\ge T^2/e$, and in particular
$M\ge(T/V)^2/4$. Lemma~\ref{lem:ideal-vm-short-interval}, applied with $T/V\ge1$ in place
of $T$, therefore gives
\[
\sum_{\Na\mathfrak c'\in[Me^{-V/T},Me^{V/T}]}\frac{\Lambda(\mathfrak c')}{(\Na\mathfrak c')^{\alpha+\beta}}
\le e^{2\eta}M^{-\alpha-\beta}\sum_{\Na\mathfrak c'\in[Me^{-V/T},Me^{V/T}]}\Lambda(\mathfrak c')
\ll\frac VT\,M^{1-\alpha-\beta}.
\]
Moreover, $M\ge1/e$ forces $\Na(\mathfrak b\mathfrak c\mathfrak d)\le ex$, while $\Na\mathfrak c\le x/y$.
Summing over $\mathfrak b,\mathfrak c,\mathfrak d$, the left-hand side of \eqref{eq:b-near} is therefore
\[
\ll\frac VT\,x^{1-\alpha-\beta}
\sum_{\mathfrak b}\frac{|s(\mathfrak b)|}{(\Na\mathfrak b)^{1-\alpha-\beta}}
\sum_{\Na\mathfrak c\le x/y}\frac{\Lambda(\mathfrak c)}{(\Na\mathfrak c)^{1-\beta}}
\sum_{\substack{\Na\mathfrak d\le ex\\ \mathfrak d\ y\text{-rough}}}\frac{|\ell(\mathfrak d)|}{\Na\mathfrak d}.
\]
By Lemma~\ref{lem:euler-products}(a) (note $\alpha+\beta\le2\eta$),
Lemma~\ref{lem:weighted-Lambda}(ii), and Lemma~\ref{lem:euler-products}(c) with $Z=ex$,
the three sums are
\[
\ll_\kappa(\log y)^\kappa,\qquad\ll(x/y)^{\beta}\log x,\qquad
\ll_\kappa\Bigl(\frac{\log x}{\log y}\Bigr)^\kappa
\]
respectively, and \eqref{eq:b-near} follows.

We now deduce \eqref{eq:R-bound}. Let $J\ge-1$ be the largest integer with $2^{J+1}\le T$.
Ideals with $|\log(x/\Na\mathfrak a)|>2^{J+1}/T$ satisfy $|\log(x/\Na\mathfrak a)|>1/2$, so for
them the minimum in $R_{\alpha,\beta}$ is at most $2/T$. By \eqref{eq:b-total}, their
contribution is $\ll_\kappa x^{1-\alpha}(\log x)^{\kappa+2}/T$. All remaining ideals satisfy
$|\log(x/\Na\mathfrak a)|\le1$, so $(x/\Na\mathfrak a)^{\sigma'}\le e^2$. Those with
$|\log(x/\Na\mathfrak a)|\le1/T$ contribute $\ll_\kappa x^{1-\alpha}(\log x)^{\kappa+1}/T$, by
\eqref{eq:b-near} with $V=1$. For $0\le j\le J$, those with
$2^j/T<|\log(x/\Na\mathfrak a)|\le2^{j+1}/T$ have minimum at most $2^{-j}$. By
\eqref{eq:b-near} with $V=2^{j+1}$, they also contribute
$\ll_\kappa x^{1-\alpha}(\log x)^{\kappa+1}/T$. There are $O(\log 2T)$ values of $j$, and $\log(2T)\ll\log x$ since $T\le\sqrt y<\sqrt x$.
This gives \eqref{eq:R-bound} and completes the proof.
\end{proof}

\subsection{Completion of the proof}

Throughout this subsection we assume the hypotheses of Proposition~\ref{prop:truncated}.
Recall that $\eta = 1/\log y\le1/6$, $c_0 = 1 + 1/\log x$, and $y \ge T^2$.

We begin with the small-prime ratio bound.

\begin{lemma}\label{lem:S-ratio}
Let $s=\sigma+it$ with $\sigma \ge 1-\eta$ and $t\in\R$. Then for all $0\le u,v\le \eta$,
\[
\left|\frac{S(s-u-v/2)}{S(s+v/2)}\right| \ll_\kappa 1.
\]
\end{lemma}

\begin{proof}
This is the bound for the ratio of the small-prime factors in the proof of
\cite[Theorem~2.5.1]{GS} (see \cite[\S2.5.2]{GS}); compare also \cite[\S2.3]{GHS19}. By
the argument of Remark~\ref{rem:Lambda-Dirichlet}, applied to each Euler factor, we have for
$\Re(z)>0$ and $\Na\mathfrak p\le y$
\[
F_{\mathfrak p}((\Na\mathfrak p)^{-z})
=\exp\Bigl(\sum_{k\ge1}\frac{\Lambda_f(\mathfrak p^k)}{k\log\Na\mathfrak p}(\Na\mathfrak p)^{-kz}\Bigr).
\]
Hence the finite product $S(z)=\prod_{\Na\mathfrak p\le y}F_{\mathfrak p}((\Na\mathfrak p)^{-z})$ is
analytic and nonvanishing on $\Re(z)>0$, and
\[
\frac{d}{dz}\log S(z)=-\sum_{\Na\mathfrak p\le y}\sum_{k\ge1}\Lambda_f(\mathfrak p^k)(\Na\mathfrak p)^{-kz}.
\]
The points $s-u-v/2$ and $s+v/2$ lie on a horizontal segment of length $u+v\le2\eta$. On
this segment $\Re(z)\ge\sigma-\tfrac32\eta\ge1-\tfrac52\eta\ge\tfrac7{12}$. For such $z$
and $\Na\mathfrak p\le y$ we have $(\Na\mathfrak p)^{-\Re z}\le e^{5/2}/\Na\mathfrak p$ and
$1-(\Na\mathfrak p)^{-\Re z}\ge1-2^{-7/12}\ge1/4$. Using also
$|\Lambda_f(\mathfrak p^k)|\le\kappa\log\Na\mathfrak p$, we obtain
\[
\Bigl|\frac{d}{dz}\log S(z)\Bigr|
\le\kappa\sum_{\Na\mathfrak p\le y}\frac{(\log\Na\mathfrak p)(\Na\mathfrak p)^{-\Re z}}{1-(\Na\mathfrak p)^{-\Re z}}
\le 4e^{5/2}\kappa\sum_{\Na\mathfrak p\le y}\frac{\log\Na\mathfrak p}{\Na\mathfrak p}\ll\kappa\log y,
\]
by Lemma~\ref{lem:weighted-Lambda}(i). Integrating along the segment gives
$|\log S(s-u-v/2)-\log S(s+v/2)|\ll\kappa\log y\cdot2\eta\ll\kappa$, and exponentiating
proves the claim.
\end{proof}

For $0\le v\le \eta$ and $t\in\mathbb R$ set
\[
P_\pm(t):=P(c_0\pm v/2+it)
=\sum_{y<\Na\mathfrak c\le x/y}
\frac{\Lambda_\ell(\mathfrak c)}{(\Na\mathfrak c)^{c_0\pm v/2}}\,
(\Na\mathfrak c)^{-it}.
\]
The next proposition estimates the truncated integral appearing in
Proposition~\ref{prop:truncated}.

\begin{proposition}\label{prop:step3-mainbound}
Let $0\le u,v\le \eta$, and write $s=c_0+it$. Then
\begin{align*}
&\left|
\frac{1}{2\pi i}\int_{c_0-iT}^{c_0+iT}
S(s-u-v/2)\,L(s+v/2)\,P(s-v/2)\,P(s+v/2)\,
\frac{x^{\,s-u-v/2}}{s-u-v/2}\,ds
\right| \\
&\qquad\ll_\kappa
x^{1-u}\,m(v)\,
\max_{|t|\le T}\left|\frac{F(c_0+v/2+it)}{c_0+v/2+it}\right|,
\end{align*}
where
\[
m(v):=
\begin{cases}
\log x, & v=0,\\[4pt]
\min\!\left(\log x,\frac1v\right), & 0<v\le \eta.
\end{cases}
\]
\end{proposition}

\begin{proof}
This is the main estimate in the proof of \cite[Theorem~2.5.1]{GS}; compare also
the proof of \cite[Theorem~1.1]{GHS19}. We include the details. For $|t|\le T$, the real parts of $s-u-v/2$ and $c_0+v/2+it$ lie in $[\tfrac12,2]$, so
$|s-u-v/2|\asymp1+|t|\asymp|c_0+v/2+it|$. Since $\Re(s+v/2)>1$, Corollary~\ref{cor:factor}
gives
\[
S(s-u-v/2)L(s+v/2)=F(s+v/2)\,\frac{S(s-u-v/2)}{S(s+v/2)},
\]
and the ratio is $\ll_\kappa1$ by Lemma~\ref{lem:S-ratio}. Since
$|x^{s-u-v/2}|=e\,x^{1-u-v/2}$, the left-hand side of the claim is
\[
\ll_\kappa x^{1-u-v/2}\max_{|t|\le T}\left|\frac{F(c_0+v/2+it)}{c_0+v/2+it}\right|
\int_{-T}^T|P_-(t)P_+(t)|\,dt .
\]
We bound the $t$-integral by Cauchy--Schwarz and Lemma~\ref{lem:MeanSquareGaussian}. Write
\[
P_\pm(t)=\sum_{y<\Na\mathfrak c\le x/y}c_\pm(\mathfrak c)\Lambda(\mathfrak c)(\Na\mathfrak c)^{-it},
\qquad
c_\pm(\mathfrak c):=\frac{\Lambda_\ell(\mathfrak c)}{\Lambda(\mathfrak c)}\,(\Na\mathfrak c)^{-c_0\mp v/2},
\]
so that $|c_\pm(\mathfrak c)|\le\kappa(\Na\mathfrak c)^{-c_0\mp v/2}$. Since $y\ge T^2$, the
support lies in $[T^2,x]$. Using $2c_0-1\ge1$, Lemma~\ref{lem:MeanSquareGaussian} and
Lemma~\ref{lem:weighted-Lambda}(ii),(i) (note $v\le\eta\le1/2$) give
\[
\int_{-T}^T|P_-(t)|^2dt\ll\kappa^2\sum_{\Na\mathfrak c\le x/y}\frac{\Lambda(\mathfrak c)}{(\Na\mathfrak c)^{1-v}}
\ll\kappa^2\Bigl(\frac xy\Bigr)^v m(v)
\]
and
\[
\int_{-T}^T|P_+(t)|^2dt\ll\kappa^2\sum_{\Na\mathfrak c\le x/y}\frac{\Lambda(\mathfrak c)}{(\Na\mathfrak c)^{1+v}}
\ll\kappa^2 m(v).
\]
Hence the $t$-integral is $\ll\kappa^2(x/y)^{v/2}m(v)$. Since
$x^{1-u-v/2}(x/y)^{v/2}=x^{1-u}y^{-v/2}\le x^{1-u}$, the proposition follows.
\end{proof}

\begin{proposition}[Integral-form ideal Hal\'asz bound]\label{prop:ideal-halasz-integral}
Let $\kappa\ge 1$ and let $f:\mathcal I\to\C$ be multiplicative. Define
\[
F(s):=\sum_{\mathfrak a\in\mathcal I}\frac{f(\mathfrak a)}{(\Na\mathfrak a)^s}\qquad(\Re(s)>1).
\]
Assume that the associated von Mangoldt coefficients $\Lambda_f$ of $f$
(as in Definition~\ref{def:Lambda_f_coeff}) satisfy
\[
|\Lambda_f(\mathfrak a)|\le \kappa\,\Lambda(\mathfrak a)
\]
for all $\mathfrak a\in\mathcal I$, where $\Lambda$ is the ideal von Mangoldt
function from Definition~\ref{def:ideal-vonmangoldt}. Then for all $x\ge 3$,
\[
\sum_{\Na\mathfrak a\le x} f(\mathfrak a)
\ \ll_{\kappa}\
\frac{x}{\log x}\int_{1/\log x}^{1}
\max_{|t|\le (\log x)^{\kappa}}
\left|\frac{F(1+\sigma+it)}{1+\sigma+it}\right|
\frac{d\sigma}{\sigma}
\ +\
\frac{x}{\log x}(\log\log x)^{\kappa}.
\]
\end{proposition}

\begin{proof}[Proof of Proposition~\ref{prop:ideal-halasz-integral}]
We follow the end of the proof of \cite[Theorem~2.5.1]{GS} (equivalently, of
\cite[Theorem~1.1]{GHS19}). We may assume $x$ is sufficiently large (depending on $\kappa$), since for bounded
$x$ the claim follows after enlarging the implied constant. Fix parameters
\[
T := (\log x)^{\kappa+3}, \qquad y := T^2,
\]
so that $y \ge \max(T^2,e^6)$ and, for large $x$, also $x>y$. Put $\eta := 1/\log y$ and
$c_0 := 1 + 1/\log x$.

Start from the truncated integral representation, Proposition~\ref{prop:truncated}:
\begin{align*}
\sum_{\Na\mathfrak a \le x} f(\mathfrak a)
&=
\int_0^\eta \int_0^\eta
\frac{1}{2\pi i}
\int_{c_0-iT}^{c_0+iT}
S(s-\alpha-\beta/2)\,L(s+\beta/2)\\
&\qquad\qquad\times
P(s-\beta/2)\,P(s+\beta/2)\,
\frac{x^{\,s-\alpha-\beta/2}}{s-\alpha-\beta/2}\,ds\,d\beta\,d\alpha \\
&\qquad\qquad
+ O\!\left(\frac{x}{\log x}(\log y)^\kappa\right)
+ O\!\left(\frac{x(\log x)^{\kappa+2}}{T\log y}\right).
\end{align*}
Taking absolute values and applying Proposition~\ref{prop:step3-mainbound} (with $u=\alpha$, $v=\beta$),
we obtain
\begin{align*}
\left|\sum_{\Na\mathfrak a \le x} f(\mathfrak a)\right|
&\ll_\kappa
\int_0^\eta \int_0^\eta
x^{1-\alpha}\, m(\beta)\,
\max_{|t|\le T}
\left|\frac{F(c_0+\beta/2+it)}{c_0+\beta/2+it}\right|
\,d\beta\,d\alpha \\
&\qquad
+ O\!\left(\frac{x}{\log x}(\log y)^\kappa\right)
+ O\!\left(\frac{x(\log x)^{\kappa+2}}{T\log y}\right).
\end{align*}

We first integrate in $\alpha$:
\[
\int_0^\eta x^{1-\alpha}\,d\alpha
= x \int_0^\eta e^{-\alpha\log x}\,d\alpha
= x\frac{1-e^{-\eta\log x}}{\log x}
\ll \frac{x}{\log x}.
\]
Hence
\begin{align*}
\left|\sum_{\Na\mathfrak a \le x} f(\mathfrak a)\right|
&\ll_\kappa
\frac{x}{\log x}\int_0^\eta
m(\beta)\,
\max_{|t|\le T}
\left|\frac{F(c_0+\beta/2+it)}{c_0+\beta/2+it}\right|
\,d\beta \\
&\qquad
+ O\!\left(\frac{x}{\log x}(\log y)^\kappa\right)
+ O\!\left(\frac{x(\log x)^{\kappa+2}}{T\log y}\right).
\end{align*}

Now set
\[
\sigma := \frac{1}{\log x} + \frac{\beta}{2},
\qquad\text{so that}\qquad
c_0+\beta/2 = 1+\sigma,\quad d\beta = 2\,d\sigma.
\]
Note that for all $\beta\in[0,\eta]$ we have
\[
m(\beta) \ll \frac{1}{\beta + 1/\log x} \asymp \frac{1}{\sigma},
\]
so $m(\beta)\,d\beta \ll d\sigma/\sigma$. Therefore
\begin{align*}
\int_0^\eta
m(\beta)\,
\max_{|t|\le T}\left|\frac{F(c_0+\beta/2+it)}{c_0+\beta/2+it}\right|\,d\beta
&\ll
\int_{1/\log x}^{\,1/\log x + \eta/2}
\max_{|t|\le T}\left|\frac{F(1+\sigma+it)}{1+\sigma+it}\right|
\frac{d\sigma}{\sigma}\\
&\le
\int_{1/\log x}^{1}
\max_{|t|\le T}\left|\frac{F(1+\sigma+it)}{1+\sigma+it}\right|
\frac{d\sigma}{\sigma}.
\end{align*}

We now reduce the $t$-range in the maximum. For $\sigma\ge 1/\log x$ we have the
trivial bound $|F(1+\sigma+it)| \ll_\kappa (\log x)^\kappa$ (e.g. by comparison with
$\zeta_{\Q(i)}(1+\sigma)^\kappa$), hence for $|t|>(\log x)^\kappa$,
\[
\left|\frac{F(1+\sigma+it)}{1+\sigma+it}\right|
\ll_\kappa \frac{(\log x)^\kappa}{|t|}
\le 1.
\]
Therefore
\[
\max_{|t|\le T}\left|\frac{F(1+\sigma+it)}{1+\sigma+it}\right|
\le
\max_{|t|\le (\log x)^\kappa}\left|\frac{F(1+\sigma+it)}{1+\sigma+it}\right|
+ O_\kappa(1),
\]
and the $O_\kappa(1)$ contributes
\[
\frac{x}{\log x}\int_{1/\log x}^1 \frac{d\sigma}{\sigma}
\ll \frac{x}{\log x}\log\log x
\ll_\kappa \frac{x}{\log x}(\log\log x)^\kappa
\qquad(\kappa\ge 1).
\]

Putting everything together, we obtain
\begin{align*}
\sum_{\Na\mathfrak a \le x} f(\mathfrak a)
&\ll_\kappa
\frac{x}{\log x}
\int_{1/\log x}^1
\max_{|t|\le (\log x)^\kappa}\left|\frac{F(1+\sigma+it)}{1+\sigma+it}\right|
\frac{d\sigma}{\sigma}\\
&\qquad
+ O\!\left(\frac{x}{\log x}(\log y)^\kappa\right)
+ O\!\left(\frac{x(\log x)^{\kappa+2}}{T\log y}\right)
+ O\!\left(\frac{x}{\log x}(\log\log x)^\kappa\right).
\end{align*}
With our choices $T=(\log x)^{\kappa+3}$ and $y=T^2$, we have $\log y \asymp \log\log x$,
so $(\log y)^\kappa \asymp (\log\log x)^\kappa$, and also
\[
\frac{x(\log x)^{\kappa+2}}{T\log y}
=
\frac{x(\log x)^{\kappa+2}}{(\log x)^{\kappa+3}\log y}
=
\frac{x}{\log x\log y}
\ll \frac{x}{\log x}.
\]
Hence all error terms are $\ll_\kappa \frac{x}{\log x}(\log\log x)^\kappa$, proving
Proposition~\ref{prop:ideal-halasz-integral}.
\end{proof}

\begin{proof}[Proof of Theorem~\ref{thm:ideal-halasz-M}]
We follow the proof of \cite[Corollary~2.5.2]{GS}; compare also the proof of
\cite[Corollary~1.2]{GHS19}. Let $c_0:=1+1/\log x$, let $V:=(\log x)^\kappa$, and define
\[
G(s):=\frac{F(s)}{s}.
\]
Consider the rectangle
\[
\mathcal R := \{u+iv:\ c_0 \le u \le 2,\ |v|\le V\}.
\]
By the maximum modulus principle, $\max_{s\in\mathcal R}|G(s)|$ occurs on $\partial\mathcal R$.

If this maximum occurs on the horizontal sides $|v|=V$, then for $u\in[c_0,2]$ we have
$|F(u\pm iV)|\ll_\kappa (\log x)^\kappa$ and $|u\pm iV|\asymp V$, hence $|G(u\pm iV)|\ll_\kappa 1$.
If the maximum occurs on the right side $u=2$, then $|F(2+it)|\ll_\kappa 1$ and $|2+it|\gg 1$,
so again $|G(2+it)|\ll_\kappa 1$. In either case,
\[
\max_{|t|\le V}\left|\frac{F(1+\sigma+it)}{1+\sigma+it}\right|
\ll_\kappa 1
\quad (1/\log x\le \sigma\le 1),
\]
and inserting this into Proposition~\ref{prop:ideal-halasz-integral} gives
\[
\sum_{\Na\mathfrak a\le x} f(\mathfrak a)
\ll_\kappa \frac{x}{\log x}\log\log x + \frac{x}{\log x}(\log\log x)^\kappa
\ll_\kappa \frac{x}{\log x}(\log\log x)^\kappa,
\]
which is already covered by the claimed bound in Theorem~\ref{thm:ideal-halasz-M}.

Thus we may assume that $\max_{s\in\mathcal R}|G(s)|$ occurs on the left side $u=c_0$.
By the definition of $M(x)$,
\[
\max_{|t|\le V}|G(c_0+it)| = e^{-M(x)}(\log x)^\kappa.
\]
Since every point $1+\sigma+it$ with $1/\log x\le \sigma\le 1$ and $|t|\le V$ lies in $\mathcal R$,
we deduce the uniform bound
\[
\max_{|t|\le V}|G(1+\sigma+it)| \le e^{-M(x)}(\log x)^\kappa.
\]
On the other hand, trivially
\[
|G(1+\sigma+it)| \le |F(1+\sigma+it)| \ll_\kappa \zeta_{\Q(i)}(1+\sigma)^\kappa \ll_\kappa \sigma^{-\kappa}.
\]
Therefore
\[
\max_{|t|\le V}\left|\frac{F(1+\sigma+it)}{1+\sigma+it}\right|
=
\max_{|t|\le V}|G(1+\sigma+it)|
\ll_\kappa \min\!\big(e^{-M(x)}(\log x)^\kappa,\ \sigma^{-\kappa}\big).
\]

Insert this into Proposition~\ref{prop:ideal-halasz-integral}. If $M(x)\le 0$, then
\[
\int_{1/\log x}^{1}
\Bigl(\max_{|t|\le V}\Bigl|\frac{F(1+\sigma+it)}{1+\sigma+it}\Bigr|\Bigr)
\frac{d\sigma}{\sigma}
\ll_\kappa
\int_{1/\log x}^{1}\sigma^{-\kappa-1}\,d\sigma
\ll_\kappa (\log x)^\kappa.
\]
Since in this case $M_+(x)=0$, this is
\[
\ll_\kappa (1+M_+(x))e^{-M_+(x)}(\log x)^\kappa.
\]
If $M(x)>0$, put
\[
\sigma_0 := e^{M(x)/\kappa}/\log x,
\]
so that $e^{-M(x)}(\log x)^\kappa=\sigma_0^{-\kappa}$. Splitting the integral at
$\sigma=\min(1,\sigma_0)$, we obtain
\begin{align*}
\int_{1/\log x}^{1}
\Bigl(\max_{|t|\le V}\Bigl|\frac{F(1+\sigma+it)}{1+\sigma+it}\Bigr|\Bigr)
\frac{d\sigma}{\sigma}
&\ll_\kappa
(\log x)^\kappa e^{-M(x)}
\int_{1/\log x}^{\min(1,\sigma_0)} \frac{d\sigma}{\sigma}
+
\int_{\sigma_0}^{\infty} \sigma^{-\kappa-1}\,d\sigma\\
&\ll_\kappa (1+M(x))e^{-M(x)}(\log x)^\kappa.
\end{align*}
Since now $M_+(x)=M(x)$, this is
$\ll_\kappa (1+M_+(x))e^{-M_+(x)}(\log x)^\kappa$.
Substituting these bounds into Proposition~\ref{prop:ideal-halasz-integral} yields
Theorem~\ref{thm:ideal-halasz-M}.
\end{proof}

\section{Pretentious reformulations}\label{sec:applications}

We first restate Theorem~\ref{thm:ideal-halasz-M} in pretentious-distance form, and then
deduce the $1$-bounded theorem stated in the introduction.

\subsection{Pretentious restatement}

\begin{lemma}[Euler product bound in pretentious distance]\label{lem:pret_bound_F}
Assume the hypotheses of Theorem~\ref{thm:ideal-halasz-M}. Let $x\ge 3$, set
\[
c_0:=1+\frac1{\log x},
\]
and let $F(s)$ be as in Theorem~\ref{thm:ideal-halasz-M}.
Then for every $t\in\R$,
\[
|F(c_0+it)|\ \ll_{\kappa}\ (\log x)^{\kappa}\exp\!\Big(-\mathbb D_{\kappa}(f,\Na^{it};x)^2\Big).
\]
\end{lemma}

\begin{proof}
For a prime ideal $\mathfrak p$, write
\[
F_{\mathfrak p}(z):=\sum_{k\ge 0} f(\mathfrak p^k)z^k .
\]
By the definition of the associated von Mangoldt coefficients,
\[
z\frac{F'_{\mathfrak p}(z)}{F_{\mathfrak p}(z)}
=
\sum_{k\ge 1}\frac{\Lambda_f(\mathfrak p^k)}{\log\Na\mathfrak p}z^k
\qquad (|z|<1).
\]
Integrating from $0$ to $z$ and using $F_{\mathfrak p}(0)=1$ gives
\[
\log F_{\mathfrak p}(z)
=
\sum_{k\ge 1}\frac{\Lambda_f(\mathfrak p^k)}{k\log\Na\mathfrak p}z^k .
\]
Taking $z=(\Na\mathfrak p)^{-s}$ and summing over prime ideals, the hypothesis
$|\Lambda_f(\mathfrak a)|\le \kappa\Lambda(\mathfrak a)$ gives, absolutely for $\Re(s)>1$,
\[
\log F(s)=
\sum_{\mathfrak p}\sum_{k\ge 1}
\frac{\Lambda_f(\mathfrak p^k)}{k\log\Na\mathfrak p\,(\Na\mathfrak p)^{ks}}.
\]
The $k=1$ term is $f(\mathfrak p)(\Na\mathfrak p)^{-s}$, so
\[
\log |F(s)|
=
\sum_{\mathfrak p}\Re\!\left(\frac{f(\mathfrak p)}{(\Na\mathfrak p)^s}\right)
+
O_\kappa\!\left(
\sum_{\mathfrak p}\sum_{k\ge 2}\frac{1}{k(\Na\mathfrak p)^{k\Re(s)}}
\right).
\]
The double sum over $k\ge 2$ is $O_\kappa(1)$ when $\Re(s)=c_0>1$. Thus
\[
\log |F(c_0+it)|
=
\sum_{\mathfrak p}\Re\!\left(\frac{f(\mathfrak p)}{(\Na\mathfrak p)^{c_0+it}}\right)
+O_\kappa(1).
\]

We now split into $\Na\mathfrak p\le x$ and $\Na\mathfrak p>x$. Since
$\log\Na\mathfrak p\ge\log x$ in the second range, Lemma~\ref{lem:PIT} and partial summation give
\[
\sum_{\Na\mathfrak p>x}\frac{1}{(\Na\mathfrak p)^{c_0}}
\le
\frac1{\log x}
\sum_{\Na\mathfrak p>x}\frac{\log\Na\mathfrak p}{(\Na\mathfrak p)^{c_0}}
\ll
\frac1{\log x}\cdot\frac{x^{1-c_0}}{c_0-1}
\ll 1.
\]
In the first range,
\[
\sum_{\Na\mathfrak p\le x}
\Re\!\left(\frac{f(\mathfrak p)}{(\Na\mathfrak p)^{c_0+it}}\right)
=
\sum_{\Na\mathfrak p\le x}
\frac{\Re\!\bigl(f(\mathfrak p)(\Na\mathfrak p)^{-it}\bigr)}{\Na\mathfrak p}
+O_\kappa(1),
\]
because $|f(\mathfrak p)|\le\kappa$ and
$1-(\Na\mathfrak p)^{-1/\log x}\ll (\log\Na\mathfrak p)/\log x$.
Therefore
\[
\log |F(c_0+it)|
=
\sum_{\Na\mathfrak p\le x}
\frac{\Re\!\bigl(f(\mathfrak p)(\Na\mathfrak p)^{-it}\bigr)}{\Na\mathfrak p}
+O_\kappa(1).
\]
By the definition of $\mathbb D_\kappa$,
\[
\sum_{\Na\mathfrak p\le x}
\frac{\Re\!\bigl(f(\mathfrak p)(\Na\mathfrak p)^{-it}\bigr)}{\Na\mathfrak p}
=
\kappa\sum_{\Na\mathfrak p\le x}\frac{1}{\Na\mathfrak p}
-
\mathbb D_\kappa(f,\Na^{it};x)^2.
\]
Using Lemma~\ref{lem:Mertens},
\[
\sum_{\Na\mathfrak p\le x}\frac{1}{\Na\mathfrak p}
=
\log\log x+O(1),
\]
so
\[
\log |F(c_0+it)|
\le
\kappa\log\log x-\mathbb D_\kappa(f,\Na^{it};x)^2+O_\kappa(1).
\]
Exponentiating completes the proof.
\end{proof}

\begin{corollary}[Pretentious form of the ideal Hal\'asz theorem]\label{cor:pret_cor2}
Assume the hypotheses of Theorem~\ref{thm:ideal-halasz-M}.
For $x\ge 3$ let
\[
M_{\mathrm{pret},\kappa}(x):=\min_{|t|\le (\log x)^\kappa}\mathbb D_\kappa(f,\Na^{it};x)^2,
\]
with $\mathbb D_\kappa$ as in Definition~\ref{def:pret_distance}.
Then
\[
\sum_{\Na\mathfrak a\le x} f(\mathfrak a)
\ \ll_\kappa\
(1+M_{\mathrm{pret},\kappa}(x))e^{-M_{\mathrm{pret},\kappa}(x)}\,x(\log x)^{\kappa-1}
\ +\ \frac{x}{\log x}(\log\log x)^\kappa .
\]
\end{corollary}

\begin{proof}[Proof of Corollary~\ref{cor:pret_cor2}]
Let $V=(\log x)^\kappa$ and $c_0=1+1/\log x$. Let $M(x)$ and $M_+(x)$ be as in
Theorem~\ref{thm:ideal-halasz-M}, so
\[
\max_{|t|\le V}\left|\frac{F(c_0+it)}{c_0+it}\right|
=e^{-M(x)}(\log x)^\kappa.
\]
By Lemma~\ref{lem:pret_bound_F} and $|c_0+it|\ge 1$,
\[
e^{-M(x)}(\log x)^\kappa
\ll_{\kappa} (\log x)^\kappa \exp\!\bigl(-M_{\mathrm{pret},\kappa}(x)\bigr),
\]
so
\[
M(x)\ge M_{\mathrm{pret},\kappa}(x)-O_{\kappa}(1).
\]
Hence also
\[
M_+(x)\ge M_{\mathrm{pret},\kappa}(x)-O_{\kappa}(1).
\]
Since \(u\mapsto (1+u)e^{-u}\) is decreasing on \([0,\infty)\), this lower bound for
\(M_+(x)\) implies, when \(M_{\mathrm{pret},\kappa}(x)\) is larger than a sufficiently large
constant depending only on \(\kappa\),
\begin{align*}
(1+M_+(x))e^{-M_+(x)}
&\le
\bigl(1+M_{\mathrm{pret},\kappa}(x)+O_\kappa(1)\bigr)
\exp\!\bigl(-M_{\mathrm{pret},\kappa}(x)+O_\kappa(1)\bigr)\\
&\ll_\kappa
(1+M_{\mathrm{pret},\kappa}(x))e^{-M_{\mathrm{pret},\kappa}(x)}.
\end{align*}
If $M_{\mathrm{pret},\kappa}(x)=O_\kappa(1)$, the same bound is absorbed into the implied constant. Thus
\[
(1+M_+(x))e^{-M_+(x)}
\ll_\kappa
(1+M_{\mathrm{pret},\kappa}(x))e^{-M_{\mathrm{pret},\kappa}(x)}.
\]
Substituting the $M_+$-form of Theorem~\ref{thm:ideal-halasz-M} gives the desired bound.
\end{proof}

\begin{lemma}[Tail estimate for \(h\)]\label{lem:h_tail}
Let $h:\mathcal I\to\C$ be multiplicative, assume that
\[
h(\mathfrak p)=0
\qquad\text{and}\qquad
|h(\mathfrak p^k)|\le 2
\quad (k\ge 2)
\]
for every prime ideal $\mathfrak p$.
Then for every fixed $\sigma\in(1/2,1)$,
\[
\sum_{\mathfrak d}\frac{|h(\mathfrak d)|}{(\Na\mathfrak d)^\sigma}\ll_\sigma 1,
\]
and consequently
\[
\sum_{\Na\mathfrak d>z}\frac{|h(\mathfrak d)|}{\Na\mathfrak d}
\ll_\sigma z^{-(1-\sigma)}
\qquad(z\ge 1).
\]
\end{lemma}

\begin{proof}
By multiplicativity,
\[
\sum_{\mathfrak d}\frac{|h(\mathfrak d)|}{(\Na\mathfrak d)^\sigma}
=
\prod_{\mathfrak p}
\left(1+\sum_{k\ge 2}\frac{|h(\mathfrak p^k)|}{(\Na\mathfrak p)^{k\sigma}}\right)
\le
\prod_{\mathfrak p}
\left(1+2\sum_{k\ge 2}\frac{1}{(\Na\mathfrak p)^{k\sigma}}\right).
\]
Since $\sigma>1/2$, the logarithm of the Euler product is
\[
\ll \sum_{\mathfrak p}\frac{1}{(\Na\mathfrak p)^{2\sigma}}<\infty,
\]
so the first claim follows.

For the tail, if $\Na\mathfrak d>z$ then
\[
\frac1{\Na\mathfrak d}
=
\frac1{(\Na\mathfrak d)^\sigma}(\Na\mathfrak d)^{-(1-\sigma)}
\le
z^{-(1-\sigma)}\frac1{(\Na\mathfrak d)^\sigma}.
\]
Summing over $\Na\mathfrak d>z$ proves the second claim.
\end{proof}

\subsection{Deduction of the $1$-bounded theorem}

\begin{proof}[Proof of Theorem~\ref{thm:halasz_1bounded_mult}]
We first treat the completely multiplicative case. If $f$ is completely multiplicative, then
\[
F(s)=\prod_{\mathfrak p}\left(1-\frac{f(\mathfrak p)}{(\Na\mathfrak p)^s}\right)^{-1}
\qquad(\Re(s)>1),
\]
and taking logarithmic derivatives shows that for prime powers $\mathfrak p^k$,
$\Lambda_f(\mathfrak p^k)=f(\mathfrak p)^k\log(\Na\mathfrak p)$.
Since $|f(\mathfrak p)|\le 1$, we have $|\Lambda_f|\le \Lambda$, so Corollary~\ref{cor:pret_cor2} with $\kappa=1$, where
$M_{\mathrm{pret},1}=M_{\mathrm{pret}}$, gives Theorem~\ref{thm:halasz_1bounded_mult} in
this case.

For general multiplicative $f$, define $g$ completely multiplicative by $g(\mathfrak p)=f(\mathfrak p)$ on prime ideals,
and define $h$ multiplicatively by $h(\Z[i])=1$, $h(\mathfrak p)=0$, and
$h(\mathfrak p^k):=f(\mathfrak p^k)-f(\mathfrak p)f(\mathfrak p^{k-1})$ for $k\ge 2$.
Then $|h(\mathfrak p^k)|\le 2$, and $h(\mathfrak d)\neq 0$ only if every prime ideal dividing $\mathfrak d$ occurs
with exponent at least $2$.

We claim that $f=g*h$. This is the ideal analogue of \cite[Ex.~2.3.4]{GS}. Since both
sides are multiplicative, it suffices to check this on prime powers. For $k\ge1$, since
$h(\mathfrak p)=0$ and $g(\mathfrak p^j)=f(\mathfrak p)^j$,
\[
(g*h)(\mathfrak p^k)=f(\mathfrak p)^k+\sum_{m=2}^{k}f(\mathfrak p)^{k-m}
\bigl(f(\mathfrak p^m)-f(\mathfrak p)f(\mathfrak p^{m-1})\bigr).
\]
The sum telescopes to $f(\mathfrak p^k)-f(\mathfrak p)^k$, so $(g*h)(\mathfrak p^k)=f(\mathfrak p^k)$.
Consequently
\[
S_f(x)=\sum_{\Na\mathfrak d\le x} h(\mathfrak d)\,S_g(x/\Na\mathfrak d),
\qquad S_g(y):=\sum_{\Na\mathfrak a\le y} g(\mathfrak a).
\]

Applying the completely multiplicative case to $g$ at level $y$ gives
\[
S_g(y)\ll (1+M_{\mathrm{pret}}(y))e^{-M_{\mathrm{pret}}(y)}\,y+\frac{y}{\log y}\log\log y.
\]
Since $f$ and $g$ agree on prime ideals, they have the same $M_{\mathrm{pret}}(\cdot)$.

Let $x^{1/2}\le y\le x$, and let $t_y$, with $|t_y|\le\log y$, attain the minimum in the
definition of $M_{\mathrm{pret}}(y)$. Since $|t_y|\le\log x$ and $|g(\mathfrak p)|\le1$,
\begin{align*}
M_{\mathrm{pret}}(x)\le\mathbb D(g,\Na^{it_y};x)^2
&=\mathbb D(g,\Na^{it_y};y)^2+\sum_{y<\Na\mathfrak p\le x}
\frac{1-\Re\bigl(g(\mathfrak p)(\Na\mathfrak p)^{-it_y}\bigr)}{\Na\mathfrak p}\\
&\le M_{\mathrm{pret}}(y)+2\sum_{y<\Na\mathfrak p\le x}\frac1{\Na\mathfrak p}.
\end{align*}
By Lemma~\ref{lem:Mertens}, the last sum is $\log\frac{\log x}{\log y}+O(1)\le\log2+O(1)$,
so $M_{\mathrm{pret}}(y)\ge M_{\mathrm{pret}}(x)-O(1)$. Since $u\mapsto(1+u)e^{-u}$ is
decreasing on $[0,\infty)$, it follows that
$(1+M_{\mathrm{pret}}(y))e^{-M_{\mathrm{pret}}(y)}\ll(1+M_{\mathrm{pret}}(x))e^{-M_{\mathrm{pret}}(x)}$.
Finally, $u\mapsto\log\log u/\log u$ is decreasing for $u\ge e^e$ and $\log y\ge\frac12\log x$,
so $\log\log y/\log y\le2\log\log x/\log x$ for $x$ sufficiently large.

Split the convolution sum into $\Na\mathfrak d\le x^{1/2}$ and $\Na\mathfrak d> x^{1/2}$. In the first range,
$y=x/\Na\mathfrak d\ge x^{1/2}$, so
\[
|S_g(x/\Na\mathfrak d)|\ll (1+M_{\mathrm{pret}}(x))e^{-M_{\mathrm{pret}}(x)}\frac{x}{\Na\mathfrak d}
+\frac{x}{\Na\mathfrak d}\cdot\frac{\log\log x}{\log x}.
\]
Moreover
\[
\sum_{\mathfrak d}\frac{|h(\mathfrak d)|}{\Na\mathfrak d}
=\prod_{\mathfrak p}\left(1+\sum_{k\ge 2}\frac{|h(\mathfrak p^k)|}{(\Na\mathfrak p)^k}\right)
\le \prod_{\mathfrak p}\left(1+2\sum_{k\ge 2}\frac1{(\Na\mathfrak p)^k}\right)\ll 1,
\]
so the contribution of $\Na\mathfrak d\le x^{1/2}$ is bounded by the right-hand side of
Theorem~\ref{thm:halasz_1bounded_mult}.

In the second range we use
\[
|S_g(Y)|\le \sum_{\Na\mathfrak a\le Y}1\ll Y\qquad(Y\ge1),
\]
which follows by counting canonical generators in the quarter disk $|z|^2\le Y$. Thus
$|S_g(x/\Na\mathfrak d)|\ll x/\Na\mathfrak d$. Applying
Lemma~\ref{lem:h_tail} with, say, $\sigma=3/4$, we obtain
\[
\sum_{\Na\mathfrak d>x^{1/2}} |h(\mathfrak d)|\,|S_g(x/\Na\mathfrak d)|
\ll
x\sum_{\Na\mathfrak d>x^{1/2}}\frac{|h(\mathfrak d)|}{\Na\mathfrak d}
\ll x\cdot x^{-1/8}
=
x^{7/8}.
\]
Since
\[
x^{7/8}\ll \frac{x}{\log x}\log\log x
\]
for all sufficiently large $x$, this tail is admissible. Enlarging the implied constant
covers the remaining bounded range of $x$. Combining the two ranges completes the proof.
\end{proof}

\section{Sectorial Hal\'asz theorem}\label{sec:sectorial_halasz}

In this section we prove the sectorial Hal\'asz theorem. The key input is a truncated
Fourier expansion for the indicator of a sector, whose coefficients contribute a factor
of $\log T$. We then apply Theorem~\ref{thm:halasz_1bounded_mult} to the multiplicative functions
$f\lambda_m$ for finitely many nonzero integers $m$.

\subsection{Canonical generators and angular characters}

Let
\[
\mathcal G:=\{z\in\Z[i]\setminus\{0\}: 0\le \arg(z)<\tfrac{\pi}{2}\}.
\]

We continue to use the angular notation from Definition~\ref{def:angular_notation}. In particular,
every nonzero ideal $\mathfrak a\subset\Z[i]$ has a unique generator
$z_{\mathfrak a}\in\mathcal G$,

\[
\arg(\mathfrak a)=\arg z_{\mathfrak a}\in[0,\tfrac{\pi}{2}),
\qquad
\lambda_m(\mathfrak a)=e^{4im\arg(\mathfrak a)}.
\]
Since $\arg(\mathfrak a\mathfrak b)\equiv \arg(\mathfrak a)+\arg(\mathfrak b)\pmod{\pi/2}$,
each $\lambda_m$ is completely multiplicative on ideals.

\subsection{Fourier approximation of sector indicators}

Write $e(t):=\exp(2\pi i t)$, and for $\theta\in\R$ define
\[
\|\theta\|_{\pi/2}:=\min_{k\in\Z}\left|\theta-k\tfrac{\pi}{2}\right|.
\]

\begin{lemma}[Truncated Fourier expansion of a sector indicator]\label{lem:fourier_angle_fixed}
Let $J=[\theta_1,\theta_2)\subset [0,\pi/2)$ and set $\delta_J=|J|/(\pi/2)$.
For each integer $T\ge 1$ there exist coefficients $b_m(J)$ for $0<|m|\le T$, depending
on $T$ (we suppress this dependence from the notation), such that for all $\theta\in[0,\pi/2)$ with $\theta\not\equiv \theta_1,\theta_2\pmod{\pi/2}$,
\[
\mathbf 1_J(\theta)
=
\delta_J+\sum_{0<|m|\le T} b_m(J)e^{4im\theta}+R_T(\theta),
\]
with
\[
|b_m(J)|\ll \frac1{|m|}\qquad(1\le |m|\le T),
\qquad
\sum_{1\le |m|\le T}|b_m(J)|\ll \log(T+1),
\]
and
\[
|R_T(\theta)|
\ll
\min\!\Bigl(
1,\,
\frac{1}{T\|\theta-\theta_1\|_{\pi/2}}
+
\frac{1}{T\|\theta-\theta_2\|_{\pi/2}}
\Bigr).
\]
\end{lemma}

\begin{proof}
Let $\psi(x):=\{x\}-\frac12$ for $x\notin\Z$, and put $\alpha:=2\theta_1/\pi$ and
$\beta:=2\theta_2/\pi$, so that $0\le\alpha<\beta<1$. For $x\in[0,1)$ with $x\ne\alpha,\beta$
we have
\[
\mathbf 1_{[\alpha,\beta)}(x)=(\beta-\alpha)+\psi(x-\beta)-\psi(x-\alpha).
\]
By a theorem of Vaaler \cite[Theorem~18]{Vaaler1985}, there is a trigonometric polynomial $\psi_T(x)=\sum_{1\le|m|\le T}a_me(mx)$ with
$|a_m|\le1/(2\pi|m|)$ such that
\[
|\psi(x)-\psi_T(x)|\le\frac{1}{2T+2}\,K_T(x)\qquad(x\notin\Z),
\]
where $K_T(x)=\sum_{|m|\le T}\bigl(1-\frac{|m|}{T+1}\bigr)e(mx)$ is the Fej\'er kernel. Since
\[
K_T(x)=\frac{1}{T+1}\Bigl(\frac{\sin(\pi(T+1)x)}{\sin(\pi x)}\Bigr)^2
\le\min\Bigl(T+1,\frac{1}{4(T+1)\|x\|^2}\Bigr),
\]
we obtain $|\psi(x)-\psi_T(x)|\le\min\bigl(1,(T\|x\|)^{-2}\bigr)\le\min\bigl(1,(T\|x\|)^{-1}\bigr)$.
Applying this to $\psi(x-\beta)$ and $\psi(x-\alpha)$ gives
\begin{align*}
\mathbf 1_{[\alpha,\beta)}(x)&=(\beta-\alpha)+\sum_{1\le|m|\le T}a_m\bigl(e(-m\beta)-e(-m\alpha)\bigr)e(mx)\\
&\qquad+O\Bigl(\min\Bigl(1,\frac{1}{T\|x-\alpha\|}+\frac{1}{T\|x-\beta\|}\Bigr)\Bigr).
\end{align*}
Now take $x=2\theta/\pi$, so that $e(mx)=e^{4im\theta}$, $\beta-\alpha=\delta_J$ and
$\|x-\alpha\|=\frac2\pi\|\theta-\theta_1\|_{\pi/2}$, and similarly for $\beta$. With
$b_m(J):=a_m(e(-m\beta)-e(-m\alpha))$ we have $|b_m(J)|\le2|a_m|\ll1/|m|$, and summing over
$1\le|m|\le T$ gives $\sum|b_m(J)|\ll\log(T+1)$.
\end{proof}

\subsection{Fourier decomposition of sectorial sums}

\begin{lemma}[Counting in angular wedges in a radial window]\label{lem:wedge_count_window}
Let $0\le \theta<\pi/2$, let $0<\Delta\le 1$, and let $0\le X<Y$.
Then
\[
\#\{z\in\mathcal G: X<\Na(z)\le Y,\ \|\arg(z)-\theta\|_{\pi/2}\le \Delta\}
\ll \Delta (Y-X)+Y^{1/2}.
\]
\end{lemma}

\begin{proof}
Let $I:=\{\phi\in[0,\pi/2):\|\phi-\theta\|_{\pi/2}\le\Delta\}$. This is a union of at most two
intervals of total length at most $2\Delta$. Let $\mathcal R$ be the region
\[
\mathcal R:=\left\{re^{i\phi}: \sqrt X<r\le \sqrt Y,\ \phi\in I\right\}.
\]
The set being counted is contained in $\mathcal R\cap\mathcal G\cap\Z[i]$.
For each lattice point $z\in \mathcal R\cap\mathcal G\cap\Z[i]$, attach the unit square
\[
Q_z:=z+[0,1]^2.
\]
These squares have disjoint interiors and lie in the $\sqrt2$-neighbourhood of $\mathcal R$, so
\[
\#(\mathcal R\cap\mathcal G\cap\Z[i])\le\operatorname{area}(\mathcal R)+\operatorname{area}(\mathcal N),
\]
where $\mathcal N$ is the set of points outside $\mathcal R$ at distance at most $\sqrt2$ from
$\mathcal R$. We have $\operatorname{area}(\mathcal R)=\frac12|I|(Y-X)\le\Delta(Y-X)$. Every point
of $\mathcal N$ lies within distance $\sqrt2$ of the boundary of $\mathcal R$, which consists of
at most four circular arcs and four line segments, each of length at most $2\sqrt Y$. Since the
$\sqrt2$-neighbourhood of a circular arc or a line segment of length $L$ has area $\ll L+1$, we
get $\operatorname{area}(\mathcal N)\ll Y^{1/2}+1$. Finally, the set being counted is empty
unless $Y\ge1$, in which case $1\le Y^{1/2}$. This proves the claim.
\end{proof}

\begin{lemma}[Summing the Fourier remainder in a radial window]\label{lem:sum_remainder_window}
Let $J=[\theta_1,\theta_2)\subset[0,\pi/2)$ and let $T\ge 2$.
Then for all $0\le X<Y$,
\[
\sum_{\substack{X<\Na\mathfrak a\le Y\\ \arg(\mathfrak a)\notin\{\theta_1,\theta_2\}}}\bigl|R_T(\arg(\mathfrak a))\bigr|
\ll \frac{(Y-X)\log(T+1)}{T}+Y^{1/2},
\]
where $R_T$ is the remainder from Lemma~\ref{lem:fourier_angle_fixed}.
\end{lemma}

\begin{proof}
By Lemma~\ref{lem:fourier_angle_fixed},
\[
|R_T(\theta)|
\ll
\min\!\Bigl(
1,\,
\frac{1}{T\|\theta-\theta_1\|_{\pi/2}}
+
\frac{1}{T\|\theta-\theta_2\|_{\pi/2}}
\Bigr)
\qquad(\theta\not\equiv\theta_1,\theta_2).
\]
We bound the sum of the right-hand side over all ideals with $X<\Na\mathfrak a\le Y$,
interpreting the minimum as $1$ when $\theta\equiv\theta_1$ or $\theta\equiv\theta_2$. It
suffices to bound the contribution from one endpoint, say $\theta_1$.
For $0<\Delta\le 1$, write
\[
A(\Delta):=\#\{X<\Na\mathfrak a\le Y:\ \|\arg(\mathfrak a)-\theta_1\|_{\pi/2}\le \Delta\}.
\]
Let $\Delta_0:=1/T$. Then
\begin{align*}
&\sum_{X<\Na\mathfrak a\le Y}
\min\!\Bigl(1,\frac{1}{T\|\arg(\mathfrak a)-\theta_1\|_{\pi/2}}\Bigr) \\
&\quad\ll A(\Delta_0)
+ \sum_{\substack{j\ge 0\\ 2^{-j}>\Delta_0}}
\frac{2^j}{T}\,
\#\Bigl\{X<\Na\mathfrak a\le Y:
2^{-j-1}<\|\arg(\mathfrak a)-\theta_1\|_{\pi/2}\le 2^{-j}\Bigr\} \\
&\quad\ll A(\Delta_0)+\sum_{\substack{j\ge 0\\ 2^{-j}>\Delta_0}}\frac{2^j}{T}\,A(2^{-j}).
\end{align*}
Applying Lemma~\ref{lem:wedge_count_window} with $\Delta=\Delta_0$ and $\Delta=2^{-j}$ gives
\[
A(\Delta_0)\ll \frac{Y-X}{T}+Y^{1/2},
\qquad
A(2^{-j})\ll 2^{-j}(Y-X)+Y^{1/2}.
\]
Therefore
\[
\begin{aligned}
&\sum_{X<\Na\mathfrak a\le Y}
\min\!\Bigl(1,\frac{1}{T\|\arg(\mathfrak a)-\theta_1\|_{\pi/2}}\Bigr) \\
&\quad\ll
\frac{Y-X}{T}+Y^{1/2}
+
\sum_{2^{-j}>\Delta_0}\frac{2^j}{T}\bigl(2^{-j}(Y-X)+Y^{1/2}\bigr),
\end{aligned}
\]
and the sum on the right is
\[
\ll \frac{(Y-X)\log(T+1)}{T}
+Y^{1/2}\sum_{2^{-j}>\Delta_0}\frac{2^j}{T}
+Y^{1/2}.
\]
Since $\Delta_0=1/T$, we have
\[
\sum_{2^{-j}>\Delta_0}\frac{2^j}{T}\ll 1.
\]
It follows that
\[
\sum_{X<\Na\mathfrak a\le Y}
\min\!\left(1,\frac{1}{T\|\arg(\mathfrak a)-\theta_1\|_{\pi/2}}\right)
\ll \frac{(Y-X)\log(T+1)}{T}+Y^{1/2}.
\]
The same estimate holds for the endpoint $\theta_2$, proving the claim.
\end{proof}

\begin{proposition}[Fourier sector decomposition in a radial window]\label{prop:sector_decomp_window}
Let $f:\mathcal I\to\C$ be multiplicative with $|f(\mathfrak a)|\le 1$.
Let $J=[\theta_1,\theta_2)\subset[0,\pi/2)$ and let $T\ge 2$.
Then for all $0\le X<Y$,
\[
\begin{aligned}
&\sum_{\substack{X<\Na\mathfrak a\le Y\\ \arg(\mathfrak a)\in J}} f(\mathfrak a)
-\delta_J\sum_{X<\Na\mathfrak a\le Y} f(\mathfrak a) \\
&\quad=
\sum_{0<|m|\le T} b_m(J)
\sum_{X<\Na\mathfrak a\le Y} f(\mathfrak a)\lambda_m(\mathfrak a)
+O\!\left(\frac{(Y-X)\log(T+1)}{T}+Y^{1/2}\right),
\end{aligned}
\]
with an absolute implied constant.
\end{proposition}

\begin{proof}
Let $\mathcal E=\mathcal E(X,Y;J)$ be the set of ideals $\mathfrak a$ with
\[
X<\Na\mathfrak a\le Y
\qquad\text{and}\qquad
\arg(\mathfrak a)\in\{\theta_1,\theta_2\}.
\]
Each boundary condition places the canonical generator on a fixed ray, so
\[
\#\mathcal E\ll Y^{1/2}.
\]
For $\mathfrak a\notin\mathcal E$, Lemma~\ref{lem:fourier_angle_fixed} gives
\[
\mathbf 1_J(\arg(\mathfrak a))
=
\delta_J+\sum_{0<|m|\le T} b_m(J)\lambda_m(\mathfrak a)+R_T(\arg(\mathfrak a)).
\]
Multiply by $f(\mathfrak a)$ and sum over $X<\Na\mathfrak a\le Y$, $\mathfrak a\notin\mathcal E$.
Restoring the omitted boundary terms contributes $O(Y^{1/2})$, and
Lemma~\ref{lem:sum_remainder_window} bounds the sum of the remainder terms by
\[
O\!\left(\frac{(Y-X)\log(T+1)}{T}+Y^{1/2}\right).
\]
This proves the stated formula.
\end{proof}

\begin{theorem}[Quantitative sectorial Hal\'asz theorem]\label{thm:sectorial_halasz_1bdd_fixed}
Let $f:\mathcal I\to\C$ be multiplicative with $|f(\mathfrak a)|\le 1$ for all ideals $\mathfrak a$.
Let $J=[\theta_1,\theta_2)\subset [0,\pi/2)$ (possibly depending on $x$) and let $T=T(x)\ge 2$ be an integer.
Let $\delta_J=|J|/(\pi/2)$, and let $M_m(x)$ be as in Definition~\ref{def:M_m_1bdd_fixed}.
Then for all $x\ge 3$,
\begin{align*}
\bigl|S_{f,J}(x)-\delta_J S_f(x)\bigr|
\ \ll\ 
&x\log(T+1)\,\max_{1\le |m|\le T}\Bigl((1+M_m(x))e^{-M_m(x)}\Bigr) \\
&\quad+\ \frac{x}{\log x}(\log\log x)\,\log(T+1)\\
&\quad+\ \frac{x\log(T+1)}{T}\\
&\quad+\ x^{1/2},
\end{align*}
with an absolute implied constant (uniform in $J$ and $T$).
\end{theorem}

\begin{proof}
Apply Proposition~\ref{prop:sector_decomp_window} with $(X,Y)=(0,x)$:
\[
S_{f,J}(x)-\delta_J S_f(x)
=
\sum_{0<|m|\le T} b_m(J)\,S_{f\lambda_m}(x)
+O\!\left(\frac{x\log(T+1)}{T}+x^{1/2}\right).
\]
Taking absolute values and using
\[
\sum_{0<|m|\le T}|b_m(J)|\ll \log(T+1)
\]
gives
\[
\bigl|S_{f,J}(x)-\delta_J S_f(x)\bigr|
\ll
\sum_{0<|m|\le T}|b_m(J)|\,|S_{f\lambda_m}(x)|
+\frac{x\log(T+1)}{T}+x^{1/2}.
\]
For each $m$, the function $f\lambda_m$ is multiplicative and $1$-bounded, so
Theorem~\ref{thm:halasz_1bounded_mult} gives
\[
S_{f\lambda_m}(x)
\ll
(1+M_m(x))e^{-M_m(x)}\,x + \frac{x}{\log x}\log\log x.
\]
Insert this estimate, take the maximum over $1\le |m|\le T$, and use the bound for
$\sum_{0<|m|\le T}|b_m(J)|$ again to obtain the stated bound.
\end{proof}

\begin{remark}[Log-power decay under angular non-pretentiousness]
\label{rem:sectorial_logpower}
Under the hypotheses of Theorem~\ref{thm:sectorial_halasz_1bdd_fixed}, let $0<A\le 1$ and assume that
\[
\min_{1\le |m|\le T} M_m(x)\ge 2A\log\log x
\]
for all sufficiently large $x$. Then
\[
\bigl|S_{f,J}(x)-\delta_J S_f(x)\bigr|
\ll_A
\frac{x\log(T+1)\log\log x}{(\log x)^A}
+\frac{x\log(T+1)}{T}
+x^{1/2}.
\]
\end{remark}

\begin{proof}[Proof of Remark~\ref{rem:sectorial_logpower}]
This is immediate from Theorem~\ref{thm:sectorial_halasz_1bdd_fixed}, since for $0<A\le 1$ the
hypothesis gives
\[
(1+M_m(x))e^{-M_m(x)}\ll_A \frac{\log\log x}{(\log x)^{2A}}
\ll_A \frac{\log\log x}{(\log x)^A},
\]
and also
\[
\frac{\log\log x}{\log x}\ll \frac{\log\log x}{(\log x)^A}.
\]
\end{proof}

We now prove Theorem~\ref{thm:sectorial_cancellation_1bdd_fixed}.

\begin{proof}[Proof of Theorem~\ref{thm:sectorial_cancellation_1bdd_fixed}]
Define $T_1(x)$ to be the largest integer $K$ satisfying
\[
2\le K\le \lfloor\log x\rfloor
\]
and
\[
\log(K+1)
\max_{1\le |m|\le K}\Bigl((1+M_m(x))e^{-M_m(x)}\Bigr)
\le \frac{1}{K},
\]
with the convention $T_1(x)=2$ if no such integer exists.
For each fixed $K\ge 2$, the hypothesis of the theorem gives
\[
\max_{1\le |m|\le K}\Bigl((1+M_m(x))e^{-M_m(x)}\Bigr)\to 0\qquad (x\to\infty),
\]
so for all sufficiently large $x$ the integer $K$ belongs to the defining set. Since also
$\lfloor\log x\rfloor\to\infty$, it follows that $T_1(x)\to\infty$ and
\[
\log(T_1(x)+1)\max_{1\le |m|\le T_1(x)}\Bigl((1+M_m(x))e^{-M_m(x)}\Bigr)
\le \frac{1}{T_1(x)}=o(1).
\]

Apply Theorem~\ref{thm:sectorial_halasz_1bdd_fixed} with $T=T_1(x)$. The first term in that
bound is $o(x)$ by construction. Since $T_1(x)\le\log x$, we have
$\log(T_1(x)+1)\ll\log\log x$, so the second term is $\ll x(\log\log x)^2/\log x=o(x)$. The
third term is $o(x)$ since $T_1(x)\to\infty$, and the fourth is $x^{1/2}=o(x)$. Hence
\[
S_{f,J}(x)=\delta_J S_f(x)+o(x).
\]
If in addition $M_0(x)\to\infty$, then Theorem~\ref{thm:halasz_1bounded_mult} applied to $f$
gives $S_f(x)=o(x)$, and therefore $S_{f,J}(x)=o(x)$.
\end{proof}

\section{Sectorial short-interval Hal\'asz theorem}\label{sec:sectorial_MR}

In this section we prove Theorem~\ref{thm:sectorial_MR}. For each
integer \(m\), we compress the ideal function \(f\lambda_m\) by norm to a
multiplicative function \(g_m\) on \(\mathbb N\). We then verify the hypotheses of
Theorem~\ref{thm:mangerel} for \(g_m\). The long average appearing in
Mangerel's theorem is shown to be negligible by applying Theorem~\ref{thm:halasz_1bounded_mult}
to \(f\lambda_m\) with an Archimedean twist. Finally, the estimates for fixed-$m$ modes are
combined with the Fourier decomposition from Section~\ref{sec:sectorial_halasz}.

\subsection{A theorem of Mangerel}

We use the following result of Mangerel. For $B\ge1$, let $d_B$ be the multiplicative
function defined by $\zeta(s)^B=\sum_{n\ge1}d_B(n)n^{-s}$ for $\Re(s)>1$, so that $d_2=d$.
For $B,C\ge1$, $0<A\le B$ and $X$ large, let $\mathcal M(X;A,B,C)$ denote the set of
multiplicative functions $g:\N\to\C$ such that the following hold.
\begin{enumerate}
\item[(i)] $|g(p)|\le B$ for all primes $p\le X$.
\item[(ii)] $|g(n)|\le d_B(n)^C$ for all $n\le X$.
\item[(iii)] For all $z_0\le z\le w\le X$,
\[
\sum_{z<p\le w}\frac{|g(p)|}{p}\ge A\sum_{z<p\le w}\frac1p-O\Bigl(\frac1{\log z}\Bigr).
\]
\end{enumerate}
Here $z_0$ is a fixed constant. For such $g$ put
\[
H(g;X):=\prod_{p\le X}\Bigl(1+\frac{(|g(p)|-1)^2}{p}\Bigr),
\]
and let $t_0=t_0(g,X)\in[-X,X]$ be a point at which
\[
t\mapsto\sum_{p\le X}\frac{|g(p)|-\Re\bigl(g(p)p^{-it}\bigr)}{p}
\]
attains its minimum on $[-X,X]$.

\begin{theorem}[{\cite[Theorem~1.7]{Mangerel}}]\label{thm:mangerel}
Let $X\ge100$ and $g\in\mathcal M(X;A,B,C)$, and put $t_0=t_0(g,X)$. Let
$10\le h_0\le X/(10H(g;X))$ and $h:=h_0H(g;X)$. Then there is $\vartheta=\vartheta(A,B)>0$ such
that
\begin{align*}
&\frac2X\int_{X/2}^X\Bigl|\frac1h\sum_{x-h<n\le x}g(n)
-\frac1h\int_{x-h}^xu^{it_0}\,du\cdot\frac2X\sum_{X/2<n\le X}g(n)n^{-it_0}\Bigr|^2dx\\
&\qquad\ll_{A,B,C}\Bigl(\Bigl(\frac{\log\log h_0}{\log h_0}\Bigr)^{A}
+\Bigl(\frac{\log\log X}{(\log X)^{\vartheta}}\Bigr)^{\min\{1,A\}}\Bigr)
\prod_{p\le X}\Bigl(1+\frac{|g(p)|-1}{p}\Bigr)^2.
\end{align*}
\end{theorem}

This follows from \cite[Theorem~1.7]{Mangerel}. Indeed, by the discussion after
\cite[(1.9)]{Mangerel}, $\mathcal M(X;A,B,C)$ is contained in the class
$\mathcal M(X;A,B,C;1,\sigma)$ to which \cite[Theorem~1.7]{Mangerel} applies, for every
$0<\sigma<\sigma_{A,B}$, where
\[
\sigma_{A,B}:=\frac A3\Bigl(1-\frac{\sin(\pi A/2B)}{\pi A/2B}\Bigr)>0.
\]
Taking $\sigma=\sigma_{A,B}/2$ gives Theorem~\ref{thm:mangerel} with
$\vartheta=\min\{1,\sigma_{A,B}/2\}/(8B+21)$.
The product on the right is the size of the trivial bound for the long average
\cite[Remark~1.8]{Mangerel}. When it is bounded, Theorem~\ref{thm:mangerel} gives an
$o(1)$ bound as soon as $h_0\to\infty$.

\subsection{The functions $g_m$}

For each $m\in\Z$, define
\[
u_m(\mathfrak a):=f(\mathfrak a)\lambda_m(\mathfrak a)
\qquad (\mathfrak a\in\mathcal I),
\]
and let
\[
g_m:=u_m^*
\]
be the norm-compression from Definition~\ref{def:norm-compression}. Thus
\[
g_m(n)=\sum_{\Na\mathfrak a=n} f(\mathfrak a)\lambda_m(\mathfrak a)
\qquad (n\ge 1).
\]
By Lemma~\ref{lem:norm-compression}, the function $g_m$ is multiplicative. Also,
\[
\sum_{x<n\le x+h} g_m(n)=S_{f\lambda_m}(x;h)
\]
for all $x\ge 1$ and $h\ge 1$. Since $|f|\le 1$ and $|\lambda_m|=1$,
\[
|g_m(n)|\le \#\{\mathfrak a\in\mathcal I:\Na\mathfrak a=n\}\le d(n),
\]
and in particular $|g_m(p)|\le 2$ for every rational prime $p$.

\begin{lemma}[Mangerel auxiliary Euler factors for $g_m$]\label{lem:gm_euler_factors}
Let $m\in\Z$, and let $g_m$ be defined above. Then uniformly in $m$ and $X\ge 3$,
\[
H(g_m;X):=\prod_{p\le X}\Bigl(1+\frac{(|g_m(p)|-1)^2}{p}\Bigr)\ \ll\ \log X,
\]
and
\[
\prod_{p\le X}\Bigl(1+\frac{|g_m(p)|-1}{p}\Bigr)\ \ll\ 1.
\]
\end{lemma}

\begin{proof}
Since $|g_m(p)|\le 2$ for primes, we have $(|g_m(p)|-1)^2\le 1$.
Therefore, using $\log(1+u)\le u$,
\[
\log H(g_m;X)=\sum_{p\le X}\log\Bigl(1+\frac{(|g_m(p)|-1)^2}{p}\Bigr)
\le\sum_{p\le X}\frac1p=\log\log X+O(1),
\]
which implies $H(g_m;X)\ll \log X$.

For the second product, first note that if $p\equiv 3\pmod 4$ then there is no prime ideal
of norm $p$ in $\Z[i]$, so $g_m(p)=0$ and the corresponding Euler factor is exactly $1-1/p$.
If $p\equiv 1\pmod 4$, then there are two prime ideals of norm $p$, hence $|g_m(p)|\le 2$ and so
\[
1+\frac{|g_m(p)|-1}{p}\le 1+\frac1p=\frac{1-p^{-2}}{1-p^{-1}}.
\]
Finally, for $p=2$ there is one prime ideal of norm $2$, so the $p=2$ factor is $\le 1$.
Therefore
\[
\prod_{p\le X}\Bigl(1+\frac{|g_m(p)|-1}{p}\Bigr)
\ll
\prod_{\substack{p\le X\\ p\equiv 1\ (4)}}\frac{1-p^{-2}}{1-p^{-1}}
\prod_{\substack{p\le X\\ p\equiv 3\ (4)}}(1-p^{-1}).
\]
The convergent product $\prod_{p\equiv 1\ (4)}(1-p^{-2})$ is bounded, and Mertens' theorem in the two
residue classes modulo $4$ (see \cite{Williams74}) gives
\[
\prod_{\substack{p\le X\\ p\equiv 1\ (4)}}(1-p^{-1})^{-1}\asymp (\log X)^{1/2},
\qquad
\prod_{\substack{p\le X\\ p\equiv 3\ (4)}}(1-p^{-1})\asymp (\log X)^{-1/2}.
\]
Multiplying these two bounds yields the claimed $O(1)$ estimate.
\end{proof}

\subsection{Bounds for the functions $f\lambda_m$}

\begin{lemma}[Long-sum bound for \(f\lambda_m\)]\label{lem:long_flambda_small}
Let $f$ be as above, fix $m\in\Z$, and let $X\ge 6$, $Z>0$, and $t_0\in\R$ satisfy
\[
\frac{X}{2}\le Z\le 2X,
\qquad
|t_0|+\log Z\le 2X.
\]
Then
\[
\sum_{\Na\mathfrak a\le Z} f(\mathfrak a)\lambda_m(\mathfrak a)(\Na\mathfrak a)^{-it_0}
\ \ll\
(1+\widetilde M_m(X))e^{-\widetilde M_m(X)}\,Z\ +\ \frac{Z}{\log Z}\,\log\log Z.
\]
Equivalently,
\[
\sum_{n\le Z} g_m(n)\,n^{-it_0}
\ \ll\
(1+\widetilde M_m(X))e^{-\widetilde M_m(X)}\,Z\ +\ \frac{Z}{\log Z}\,\log\log Z.
\]
\end{lemma}

\begin{proof}
Apply Theorem~\ref{thm:halasz_1bounded_mult} to the ideal function
\[
F_{m,t_0}(\mathfrak a):=f(\mathfrak a)\lambda_m(\mathfrak a)(\Na\mathfrak a)^{-it_0},
\]
which is multiplicative and $1$-bounded.

Its pretentious parameter at scale $Z$ is
\[
\mathcal M_{m,t_0}(Z)
:=
\min_{|u|\le \log Z}\ \sum_{\Na\mathfrak p\le Z}
\frac{1-\Re\!\Bigl(f(\mathfrak p)\lambda_m(\mathfrak p)(\Na\mathfrak p)^{-i(t_0+u)}\Bigr)}{\Na\mathfrak p}.
\]
For every $|u|\le \log Z$, the hypothesis $|t_0|+\log Z\le 2X$ gives
\[
|t_0+u|\le 2X.
\]
Since $Z\in[X/2,2X]$, Lemma~\ref{lem:Mertens} implies
\[
\sum_{\min\{X,Z\}<\Na\mathfrak p\le \max\{X,Z\}}\frac{1}{\Na\mathfrak p}\ll 1.
\]
Therefore
\begin{align*}
\mathbb D_m(f;t_0+u;Z)^2
&\ge
\mathbb D_m(f;t_0+u;X)^2
-2\!\!\sum_{\min\{X,Z\}<\Na\mathfrak p\le \max\{X,Z\}}\frac{1}{\Na\mathfrak p}\\
&\ge
\widetilde M_m(X)-O(1).
\end{align*}
Thus
\[
\mathcal M_{m,t_0}(Z)\ge \widetilde M_m(X)-O(1).
\]

Insert this lower bound into Theorem~\ref{thm:halasz_1bounded_mult}. The $O(1)$ shift is harmless, since if
$\mathcal M_{m,t_0}(Z)\ge \widetilde M_m(X)-C$ and $\mathcal M_{m,t_0}(Z)\ge 0$, then
\[
(1+\mathcal M_{m,t_0}(Z))e^{-\mathcal M_{m,t_0}(Z)}
\ll_C
(1+\widetilde M_m(X))e^{-\widetilde M_m(X)}.
\]
Finally, by the definition of $g_m$,
\[
\sum_{\Na\mathfrak a\le Z} f(\mathfrak a)\lambda_m(\mathfrak a)(\Na\mathfrak a)^{-it_0}
=\sum_{n\le Z} g_m(n)n^{-it_0}.
\]
\end{proof}

\begin{proposition}[Short-interval bound for a fixed angular mode]\label{prop:fixed_flambda_short_interval}
Let $h=h(X)$ satisfy $h=o(X)$ and $h/\sqrt X\to\infty$. Fix an integer
$m\ne0$. Assume hypotheses \textup{(H1)} and \textup{(H2)} of
Theorem~\ref{thm:sectorial_MR} for this $m$. Then
\[
E_m(X):=
\left(\frac{2}{X}\int_{X/2}^{X}\left|\frac{S_{f\lambda_m}(x;h)}{h}\right|^2dx\right)^{1/2}
=o(1).
\]
\end{proposition}

\begin{proof}
Set $Y:=X+h$. We first verify the hypotheses of Theorem~\ref{thm:mangerel}
for the norm-compressed function $g_m$.
By construction,
\[
\sum_{x<n\le x+h} g_m(n)=S_{f\lambda_m}(x;h),
\]
while $g_m$ is multiplicative, $|g_m(p)|\le 2$, and $|g_m(n)|\le d(n)$.
For an odd rational prime $p$, if $p\equiv 1\pmod 4$ and $\mathfrak p,\overline{\mathfrak p}$
are the two prime ideals above $p$, then
\[
g_m(p)=f(\mathfrak p)\lambda_m(\mathfrak p)+f(\overline{\mathfrak p})\lambda_m(\overline{\mathfrak p}),
\]
whereas $g_m(p)=0$ for $p\equiv 3\pmod 4$. For $p=2$ one has
$g_m(2)=f((1+i))\lambda_m((1+i))$, but this prime never appears in the ranges
$z<p\le w$ because $z\ge 2$ and the left endpoint is strict.
Therefore Hypothesis~\textup{(H1)} yields
\[
2\sum_{\substack{z<p\le w\\ p\equiv 1\ (4)}}\frac{|g_m(p)|}{p}
\ge
2A_m\sum_{\substack{z<p\le w\\ p\equiv 1\ (4)}}\frac{1}{p}
-\frac{C_m}{\log z}
\]
uniformly for $2\le z\le w\le Y$. Dividing by $2$ and using Mertens' theorem in the
residue classes modulo $4$ (see \cite{Williams74}),
\[
\sum_{\substack{z<p\le w\\ p\equiv 1\ (4)}}\frac{1}{p}
=
\frac12\sum_{z<p\le w}\frac{1}{p}
+O\!\Bigl(\frac{1}{\log z}\Bigr),
\]
we obtain
\[
\sum_{z<p\le w}\frac{|g_m(p)|}{p}
\ge
\frac{A_m}{2}\sum_{z<p\le w}\frac{1}{p}
-O_m\!\Bigl(\frac{1}{\log z}\Bigr),
\]
because $g_m(p)=0$ for $p\equiv 3\pmod 4$ and $p=2$ is absent from the interval.
Put
\[
\alpha_m:=\min\{A_m/2,1\}.
\]
We now check that $g_m\in\mathcal M(Y;\alpha_m,2,1)$. Conditions~\textup{(i)} and~\textup{(ii)}
of Theorem~\ref{thm:mangerel} hold since $|g_m(p)|\le2$ and $|g_m(n)|\le d(n)=d_2(n)$.
Since $\alpha_m\le A_m/2$, the bound above gives condition~\textup{(iii)} with
$A=\alpha_m\le2=B$.
Let
\[
H_m(U):=H(g_m;U)
\qquad (U\ge 3).
\]
By Lemma~\ref{lem:gm_euler_factors}, $H_m(U)\ll \log U$. Set
\[
h_{0,m}:=\frac{h}{H_m(Y)}.
\]
Since $h/\sqrt X\to\infty$, in particular $h/\log X\to\infty$, and $H_m(Y)\ll\log Y\asymp\log X$,
we have $h_{0,m}\to\infty$. Also, since $h=o(X)$ and $Y\asymp X$, for $X$ large we have
$10\le h_{0,m}\le Y/(10H_m(Y))$.
This is the $h_0$-parameter in Theorem~\ref{thm:mangerel}.

Define
\[
\mathcal A_m(y):=\frac1h\sum_{y-h<n\le y} g_m(n).
\]
By the change of variables $y=x+h$,
\[
\frac{2}{X}\int_{X/2}^{X}\left|\frac{S_{f\lambda_m}(x;h)}{h}\right|^2dx
=
\frac{2}{X}\int_{X/2+h}^{X+h}|\mathcal A_m(y)|^2\,dy.
\]
Since $h=o(X)$, the interval $[X/2+h,X+h]$ is contained in $[Y/2,Y]$ for $X$ large, so
\[
\frac{2}{X}\int_{X/2}^{X}\left|\frac{S_{f\lambda_m}(x;h)}{h}\right|^2dx
\ll
\frac{2}{Y}\int_{Y/2}^{Y}|\mathcal A_m(y)|^2\,dy.
\]
We apply Theorem~\ref{thm:mangerel} to $g_m$ at scale $Y$, with
$t_{0,m}:=t_0(g_m,Y)\in[-Y,Y]$. The product in its error term is $O(1)$ by the second part
of Lemma~\ref{lem:gm_euler_factors}, and $h_{0,m}\to\infty$, so the error term is $o(1)$.
This gives
\[
\frac{2}{Y}\int_{Y/2}^{Y}\left|\mathcal A_m(y)
-\Bigl(\frac{1}{h}\int_{y-h}^{y}u^{it_{0,m}}\,du\Bigr)\cdot
\frac{2}{Y}\sum_{Y/2<n\le Y} g_m(n)n^{-it_{0,m}}
\right|^2dy
\ = \ o(1).
\]

It remains to bound the main term, which we do using Lemma~\ref{lem:long_flambda_small} and
Hypothesis~\textup{(H2)}. Since $h=o(X)$, we have $Y=X+h=(1+o(1))X$, and therefore
\[
Y+\log Y\le 2X
\]
for all sufficiently large $X$. Hence
\[
|t_{0,m}|+\log Y\le Y+\log Y\le 2X.
\]
Applying Lemma~\ref{lem:long_flambda_small} with this $t_0=t_{0,m}$ gives
\[
\sum_{n\le Y} g_m(n)n^{-it_{0,m}}
\ll
(1+\widetilde M_m(X))e^{-\widetilde M_m(X)}\,Y
+\frac{Y}{\log Y}\, \log\log Y
=
o(Y),
\]
by Hypothesis~\textup{(H2)}.
Applying the same lemma with $Z=Y/2$ also gives
\[
\sum_{n\le Y/2} g_m(n)n^{-it_{0,m}}=o(Y).
\]
Subtracting, we obtain
\[
\sum_{Y/2<n\le Y} g_m(n)n^{-it_{0,m}}=o(Y),
\]
and hence
\[
\frac{2}{Y}\sum_{Y/2<n\le Y} g_m(n)n^{-it_{0,m}}=o(1).
\]
Since
\[
\left|\frac{1}{h}\int_{y-h}^{y}u^{it_{0,m}}\,du\right|\le 1,
\]
the main term above is $o(1)$. Combining this with the error term from Theorem~\ref{thm:mangerel} yields
the stated estimate.
\end{proof}

\subsection{Proof of the sectorial short-interval Hal\'asz theorem}

\begin{proof}[Proof of Theorem~\ref{thm:sectorial_MR}]
For each fixed integer $m\ne0$, Proposition~\ref{prop:fixed_flambda_short_interval} gives
$E_m(X)=o(1)$. Hence there exists an integer-valued function $T_1(X)\to\infty$ such that
\[
\log(T_1(X)+1)\max_{1\le |m|\le T_1(X)}E_m(X)\to0
\]
and
\[
\frac{\log(T_1(X)+1)}{T_1(X)}\to0.
\]
Indeed, choose $X_j\to\infty$, $j\ge2$, so large that, for all $X\ge X_j$,
\[
\max_{1\le |m|\le j}E_m(X)\le \frac1{j\log(j+1)}.
\]
After increasing the sequence if necessary so that it is strictly increasing, set
$T_1(X)=j$ for $X_j\le X<X_{j+1}$, and set $T_1(X)=2$ for $X<X_2$.

Apply Proposition~\ref{prop:sector_decomp_window} with $(X,Y)=(x,x+h)$ and $T_1(X)$ in place of $T$:
\begin{align*}
S_{f,J}(x;h)-\delta_J S_f(x;h)
&=
\sum_{0<|m|\le T_1(X)} b_m(J)\,S_{f\lambda_m}(x;h) \\
&\quad+O\!\left(\frac{h\log(T_1(X)+1)}{T_1(X)}+X^{1/2}\right),
\end{align*}
uniformly for $x\in[X/2,X]$.
Divide by $h$ and take $L^2$-means over $x\in[X/2,X]$.
Using $\sum_{0<|m|\le T_1(X)}|b_m(J)|\ll\log(T_1(X)+1)$ and the triangle inequality in $L^2$, we get
\begin{align*}
\left(\frac{2}{X}\int_{X/2}^{X}\left|\frac{S_{f,J}(x;h)-\delta_J S_f(x;h)}{h}\right|^2dx\right)^{1/2}
&\ll \log(T_1(X)+1)\max_{1\le |m|\le T_1(X)}E_m(X) \\
&\quad + \frac{\log(T_1(X)+1)}{T_1(X)}+\frac{X^{1/2}}{h}.
\end{align*}
The right-hand side is $o(1)$, by the choice of $T_1(X)$ and the hypothesis $h/\sqrt X\to\infty$. Therefore
\[
\frac{2}{X}\int_{X/2}^{X}
\left|\frac{S_{f,J}(x;h)-\delta_J S_f(x;h)}{h}\right|^2dx
=o(1).
\]
\end{proof}

\begin{proof}[Proof of Remark~\ref{rem:sectorial_MR_unrestricted}]
The proof of Proposition~\ref{prop:fixed_flambda_short_interval} applies with $m=0$ under the
additional hypotheses in Remark~\ref{rem:sectorial_MR_unrestricted}. Since $\lambda_0\equiv 1$,
\[
g_0(n)=\sum_{\Na\mathfrak a=n}f(\mathfrak a),
\qquad
\sum_{x<n\le x+h} g_0(n)=S_f(x;h).
\]
Therefore
\[
\frac{2}{X}\int_{X/2}^{X}\left|\frac{S_f(x;h)}{h}\right|^2\,dx=o(1).
\]
The pointwise conclusion outside a set of Lebesgue measure $o(X)$ follows from Chebyshev's inequality.
\end{proof}

\section*{Acknowledgments}

The author thanks his advisor, Joel Moreira, for valuable feedback during the development of this work.

Jan Ku\'s is supported by the Warwick Mathematics Institute Centre for Doctoral Training, and gratefully
acknowledges funding from the European Research Council (ERC) under the European Union's Horizon 2020
research and innovation programme (Grant agreement No.~EP/Y014030/1).

\appendix

\section{Auxiliary estimates for prime ideals}\label{app:aux}

In this appendix we prove the prime-ideal sums, Euler-product bounds and the mean-value
inequality used in Section~\ref{sec:ideal-halasz-proof}. They follow from
Lemmas~\ref{lem:PIT} and~\ref{lem:Mertens}, and their analogues over $\N$ are standard (see, for
instance, \cite[Chapters~1.2 and~2.4]{GS}). As in Section~\ref{sec:ideal-halasz-proof},
$C_0$ denotes the constant in \eqref{eq:C0}.

\begin{lemma}[Weighted prime-ideal sums]\label{lem:weighted-Lambda}
Let $X\ge2$.
\begin{enumerate}
\item[(i)] If $0\le\beta\le1/2$, then
\[
\sum_{\Na\mathfrak a\le X}\frac{\Lambda(\mathfrak a)}{(\Na\mathfrak a)^{1+\beta}}
\ll\min\Bigl(\log X,\frac1\beta\Bigr).
\]
\item[(ii)] If $0\le\beta\le1/2$, then
\[
\sum_{\Na\mathfrak a\le X}\frac{\Lambda(\mathfrak a)}{(\Na\mathfrak a)^{1-\beta}}
\ll X^{\beta}\min\Bigl(\log X,\frac1\beta\Bigr).
\]
\item[(iii)] If $0\le\alpha\le1/2$, then
\[
\sum_{\Na\mathfrak a\le X}\frac{\Lambda(\mathfrak a)}{(\Na\mathfrak a)^{\alpha}}\ll X^{1-\alpha}.
\]
\end{enumerate}
Here $1/\beta:=+\infty$ when $\beta=0$, and the implied constants are absolute.
\end{lemma}

\begin{proof}
By partial summation, for every $\gamma\in\R$,
\[
\sum_{\Na\mathfrak a\le X}\frac{\Lambda(\mathfrak a)}{(\Na\mathfrak a)^{\gamma}}
=\frac{\psi_{\Q(i)}(X)}{X^{\gamma}}+\gamma\int_1^X\frac{\psi_{\Q(i)}(u)}{u^{1+\gamma}}\,du .
\]
For $\gamma=1+\beta$, \eqref{eq:C0} bounds the right-hand side by
\[
\ll X^{-\beta}+\int_1^Xu^{-1-\beta}\,du\ll\min\Bigl(\log X,\frac1\beta\Bigr),
\]
which proves (i). For $\gamma=1-\beta$, the right-hand side is
$\ll X^{\beta}+\int_1^Xu^{-1+\beta}\,du$. The integral equals $\log X$ if $\beta=0$, and
$(X^\beta-1)/\beta\le X^\beta\min(\log X,1/\beta)$ if $\beta>0$. Since
$\min(\log X,1/\beta)\ge\log2$, this proves (ii). Finally, for $\gamma=\alpha$ the right-hand
side is $\ll X^{1-\alpha}+\alpha\int_1^Xu^{-\alpha}\,du\ll X^{1-\alpha}$, which proves (iii).
\end{proof}

\begin{lemma}[Euler-product bounds]\label{lem:euler-products}
Let $\kappa\ge1$.
\begin{enumerate}
\item[(a)] If $y\ge e^6$ and $\alpha\le3/\log y$, then
\[
\prod_{\Na\mathfrak p\le y}\Bigl(1-\frac1{(\Na\mathfrak p)^{1-\alpha}}\Bigr)^{-\kappa}
\ll_\kappa(\log y)^\kappa.
\]
\item[(b)] If $z\ge3$, then
\[
\sum_{\Na\mathfrak p>z}\frac1{(\Na\mathfrak p)^{1+1/\log z}}\ll1
\qquad\text{and}\qquad
\prod_{\Na\mathfrak p>z}\Bigl(1-\frac1{(\Na\mathfrak p)^{1+1/\log z}}\Bigr)^{-\kappa}\ll_\kappa1.
\]
\item[(c)] If $3\le y\le x$ and $c_0=1+1/\log x$, then
\[
\prod_{\Na\mathfrak p>y}\Bigl(1-\frac1{(\Na\mathfrak p)^{c_0}}\Bigr)^{-\kappa}
\ll_\kappa\Bigl(\frac{\log x}{\log y}\Bigr)^\kappa,
\]
and for every $Z\ge y$,
\[
\prod_{y<\Na\mathfrak p\le Z}\Bigl(1-\frac1{\Na\mathfrak p}\Bigr)^{-\kappa}
\ll_\kappa\Bigl(\frac{\log Z}{\log y}\Bigr)^\kappa.
\]
\end{enumerate}
\end{lemma}

\begin{proof}
(a) Each factor is increasing in $\alpha$, so it suffices to consider
$0\le\alpha\le3/\log y$, and then $\alpha\le1/2$. Put $u=1/\Na\mathfrak p$ and
$v=(\Na\mathfrak p)^{-(1-\alpha)}$. For $\Na\mathfrak p\le y$ we have $v\le2^{-1/2}$ and
$(\Na\mathfrak p)^\alpha\le e^3$, so
\[
0\le v-u=u\bigl((\Na\mathfrak p)^{\alpha}-1\bigr)\le e^3\alpha\,\frac{\log\Na\mathfrak p}{\Na\mathfrak p}.
\]
By the mean value theorem,
$\log(1-v)^{-1}-\log(1-u)^{-1}\le(v-u)/(1-2^{-1/2})$. Summing over $\Na\mathfrak p\le y$ and
using Lemma~\ref{lem:weighted-Lambda}(i) with $\beta=0$, we get
\[
\sum_{\Na\mathfrak p\le y}\Bigl(\log\frac1{1-v}-\log\frac1{1-u}\Bigr)\ll\alpha\log y\ll1.
\]
The claim now follows from Lemma~\ref{lem:Mertens}.

(b) For $j\ge1$, Lemma~\ref{lem:Mertens} gives
\[
\sum_{z^j<\Na\mathfrak p\le z^{j+1}}\frac1{\Na\mathfrak p}=\log\frac{j+1}{j}+O(1)\ll1,
\]
and $(\Na\mathfrak p)^{-1/\log z}\le e^{-j}$ in this range. Summing over $j\ge1$ proves the
first claim. The second follows since $\log(1-w)^{-1}\le2w$ for $0\le w\le1/2$.

(c) Since $\log(1-w)^{-1}\le w+w^2$ for $0\le w\le1/2$, the logarithm of the first product
is at most
\[
\kappa\sum_{y<\Na\mathfrak p\le x}\frac1{\Na\mathfrak p}
+\kappa\sum_{\Na\mathfrak p>x}\frac1{(\Na\mathfrak p)^{c_0}}
+\kappa\sum_{\mathfrak p}\frac1{(\Na\mathfrak p)^2}
=\kappa\log\frac{\log x}{\log y}+O_\kappa(1),
\]
by Lemma~\ref{lem:Mertens} and part (b) with $z=x$. The second product is bounded in the same
way.
\end{proof}

The next proposition is the ideal analogue of the case $B=0$ of
\cite[Proposition~1.2.2]{GS}, and we follow its proof.

\begin{proposition}\label{prop:GS-mean-value}
Let $g:\mathcal I\to[0,\infty)$ be multiplicative, with associated von Mangoldt coefficients
$\Lambda_g$ as in Definition~\ref{def:Lambda_f_coeff}. Assume that $\Lambda_g$ is
real-valued and that there is $A\ge0$ such that
\[
\sum_{\Na\mathfrak m\le z}\Lambda_g(\mathfrak m)\le A\,z\qquad(z\ge1).
\]
Then for every $x\ge2$,
\[
\sum_{\Na\mathfrak a\le x}g(\mathfrak a)
\le\frac{(A+1)\,x}{\log x}\sum_{\Na\mathfrak a\le x}\frac{g(\mathfrak a)}{\Na\mathfrak a}.
\]
\end{proposition}

\begin{proof}
We have
\[
(\log x)\sum_{\Na\mathfrak a\le x}g(\mathfrak a)
=\sum_{\Na\mathfrak a\le x}g(\mathfrak a)\log\Na\mathfrak a
+\sum_{\Na\mathfrak a\le x}g(\mathfrak a)\log\frac{x}{\Na\mathfrak a}.
\]
Since $g\ge0$, Lemma~\ref{lem:Lambda-convolution} and the hypothesis give
\[
\sum_{\Na\mathfrak a\le x}g(\mathfrak a)\log\Na\mathfrak a
=\sum_{\Na\mathfrak r\le x}g(\mathfrak r)\sum_{\Na\mathfrak m\le x/\Na\mathfrak r}\Lambda_g(\mathfrak m)
\le A\,x\sum_{\Na\mathfrak r\le x}\frac{g(\mathfrak r)}{\Na\mathfrak r}.
\]
Since $\log t\le t$ for $t\ge1$, the second sum is at most
$x\sum_{\Na\mathfrak a\le x}g(\mathfrak a)/\Na\mathfrak a$. Combining the two bounds proves the
proposition.
\end{proof}

\section*{Declaration of Generative AI use}

During the preparation of this work the author used Claude (Anthropic) and ChatGPT (OpenAI)
for help in editing, reviewing and literature search. The author reviewed and edited the final version and takes full
responsibility for the content of the publication.

\bibliographystyle{amsplain}
\bibliography{references}

\end{document}